\documentclass[times,review,10pt]{elsarticle}
\usepackage{pgfplots}
\pgfplotsset{compat=1.18}
\usepackage[utf8]{inputenc}
\usepackage{makecell}
\usepackage{longtable}
\usepackage{algorithm}
\usepackage{algorithmic}
\usepackage[a4paper, total={6in, 9in}]{geometry}
\usepackage{amsmath}
\usepackage{amsthm}
\usepackage{amssymb}
\usepackage{float} 
\usepackage{url}
\usepackage{graphicx}
\usepackage{caption}
\usepackage{subcaption}
\usepackage{hyperref}
\usepackage{xcolor}
\usepackage{tikz}
\theoremstyle{definition}
\newtheorem{definition}{Definition}[section]
\newtheorem{example}{Example}[section]
\newtheorem{theorem}{Theorem}[section]
\newtheorem{corollary}{Corollary}[section]
\newtheorem{proposition}{Proposition}[section]
\newtheorem{remark}{Remark}[section]
\title{Roughness and entropy measures of a soft set }
\author{Santanu Acharjee$^{1,2}$, Sankar K. Pal$^{2,3}$\\
$^{1}$Department of Mathematics\\
Gauhati University\\
Guwahati-781014, Assam, India\\
$^{2}$Center for Soft Computing Research\\
Indian Statistical Institute\\
Kolkata-700108, West Bengal,  India\\
$^{3}$Computer Science Department\\
IIIT Bhubaneswar \\
Bhubaneswar-751003, Odisha,  India\\
e-mails:$^{1}$sacharjee326@gmail.com, 
$^{3}$sankarpal@yahoo.com\\
Orcid ids: $^{1}$0000-0003-4932-3305,  $^{3}$0000-0003-3301-4751}
\date{}

\begin{document}

\maketitle

\noindent
{\bf Abstract:} Soft set theory is an important and emerging area within soft computing, owing to its attribute-oriented mathematical framework and its wide applicability in diverse domains, including science and social sciences. The theoretical constraints associated with the selection of subsets of the sets of attributes in soft set theory have further motivated the development of hybrid and extended theoretical models. In this paper, we introduce two distinct roughness measures and  six entropy measures for soft sets and systematically investigate their properties using both theoretical analysis and computational techniques. The proposed roughness measures are defined within two distinct conceptual frameworks. Throughout the development of these measures and the corresponding results, the foundational principles of soft set theory, as established by Molodtsov, are strictly preserved.  Furthermore, the proposed framework is shown to be novel with respect to roughness characterization, and a comparative analysis with classical rough set theory is presented to highlight the theoretical distinctions and contributions of this work.\\


\noindent
{\bf 2020 AMS Classifications:}  68U01; 68Q07; 68U35; 68T10.\\

\section{Introduction}
The presence of various types of uncertainties in nature has restricted human beings from achieving exact solutions to many of their problems. Thus, many experts have tried to obtain approximate solutions to their problems, and on many occasions these approximation techniques are found to be highly useful, although they are incapable of giving exact solutions. Before 1965, experts primarily addressed uncertainty in natural phenomena using probability theory and interval mathematics. This situation changed with Zadeh’s introduction of fuzzy set theory \cite{1}. Its successful applications in real-world areas such as pattern classification \cite{2}, clustering \cite{3}, and related fields immediately attracted the attention of experts from diverse domains of human knowledge. In general, the primary importance of clustering lies in extracting meaningful information or knowledge from data. Information granulation is crucial for problem-solving. Unlike philosophy and human cognition, granularity has recently found various important roles in economic decision-making \cite{4}. Zadeh \cite{5} characterized human cognition as comprising three stages: granulation, organization, and causation. Granulation focuses on decomposing a whole into its constituent parts; organization integrates these parts into a coherent whole; and causation deals with associating causes that produce corresponding effects. Although it is easy to state that granulation plays various major roles in the processes of human cognition and artificial intelligence, developing computational techniques for granulation remains a challenging task unless soft computing approaches such as fuzzy set theory, rough set theory, and related methods are employed as tools for granulation. Granulation allows for more efficient extraction of knowledge or information, particularly when data or information about an object is uncertain, incomplete, or vague. In a granule, all elements are generally considered indiscernible from each other. \\

\noindent
One of the  theoretical frameworks for performing granulation and subsequent knowledge extraction is rough set theory. Introduced by Pawlak \cite{6} in 1981, rough set theory addresses ambiguity arising from the limited discernibility of objects within the domain of discourse. It is based on the concepts of approximation spaces and indiscernibility relations. The primary objective of rough set theory is to generate logical rules for classification and prediction in uncertain, incomplete, or vague environments \cite{7}. Thus, rough set theory has been widely applied in various areas such as pattern recognition \cite{8}, image processing \cite{9}, knowledge discovery \cite{10}, big data analytics \cite{11}, and many other fields. Although rough set theory is a highly prominent and useful theoretical framework, it has certain theoretical limitations in practical applications, particularly when the overlapping of two or more objects occurs in areas such as pattern recognition and image processing. The existing definitions of lower and upper approximations in rough set theory do not allow us to clearly distinguish the lower and upper approximations of two objects in an image when they overlap. This limitation motivates us to develop a novel theory in this paper to address situations in which lower and upper approximations can be distinctly defined, thereby providing greater clarity in applications such as pattern recognition and image processing.  \\

\noindent
To address the aforementioned real life problems related to rough set theory, we adopt soft set theory as the fundamental theoretical foundation of this work. Soft set theory is a promising area of soft computing, introduced by Molodtsov \cite{12} in his seminal English language paper in 1999. However, the conceptual foundation of soft set theory had already been laid by him during the 1980s \cite{13, 14} in Russian language. It is an attribute based theory that performs approximation within the power set of a universal set through the association of attributes. The flexibility of attributes allows them to be represented in various forms, such as words, sentences, colors, open and closed intervals of real numbers,  etc. Although it may appear similar to multivalued set theory, soft set theory distinguishes itself from other existing mathematical frameworks through its underlying operations and functions \cite{13,14, 15, 16}. Moreover, Molodtsov showed that a fuzzy set is also a soft set \cite{12}, but the converse is not true. \\

\noindent
Since soft set theory associates attributes with approximate subsets of a universal set, defining a roughness measure for soft sets provides substantial potential for analyzing real-life problems, especially in scenarios involving overlapping objects. Moreover, from a philosophical perspective, it is more meaningful to distinguish objects based on the attributes connected to them, rather than treating them as an undivided whole separated from their environment. Soft set theory has many applications. Some of them are in rational analysis \cite{17}, portfolio control \cite{18}, game theory \cite{12}, topology \cite{12}, and many others. However, Molodtsov identified several conceptual errors in definitions and notions proposed by various researchers at different times and subsequently corrected some of them \cite{14, 15,16}. For instance, according to Molodtsov, the concept of a soft subset within the framework of soft set theory has no scope for existence, since it contradicts the fundamental principles of fuzzy set theory when a fuzzy set is interpreted as a soft set \cite{12} through the use of $\alpha$-cuts, where $\alpha \in [0,1]$. Moreover, he expressed concern about the development of new hybrid theories that combine soft sets with fuzzy sets or their generalizations unless the underlying mathematical and philosophical foundations are properly preserved \cite{15, 16, 19}. On the other hand, only a limited number of studies have explored soft set theory in connection with rough set-theoretic concepts; however, these approaches are not fully consistent with the mathematical and philosophical foundations of soft set theory as established by Molodtsov. Consequently, these approaches are unable to adequately study roughness measures of soft sets, thereby limiting the applicability of the resulting theoretical findings to computational domains such as pattern recognition, image processing, and clustering. Feng et al. \cite{20} proposed the notions of rough soft sets and soft rough sets and investigated some of their properties; however, their definitions and operations deviate from the original theoretical and philosophical foundations of soft set theory introduced by Molodtsov. Similar facts can be observed in Ali \cite{21}, Shabir et al. \cite{22}, Alcantud et al. \cite{23}. Moreover, none of their theoretical formulations were developed on the collection of approximate subsets of attributes within soft set theory, which constitutes a fundamental concept originally proposed by Molodtsov \cite{12,13,14,15,16,18, 19}. On numerous occasions, concepts from rough set theory were either defined on subsets of the universal set with respect to a soft set, or the corresponding operations were applied incorrectly.\\

\noindent
From the aforementioned discussions, it is evident that several uncertain situations still persist in various real-life scenarios whenever an overlap between two subsets of a universal set occurs. These uncertainties highlight the limitations of existing approaches and indicate the need for new and conceptually sound mathematical tools to address such complexities. Motivated by the growing demand to effectively handle real-life problems in areas such as pattern recognition, image processing, clustering, and related fields due to rough set theoretical limitations, the development of improved theoretical frameworks becomes both necessary and timely. Hence, in this paper, we introduce the concept of the soft rough set by adhering to the conceptually sound mathematical foundations of soft set theory as established by Molodtsov in his seminal works \cite{12,13,15,16,17, 18}.\\

\noindent
This paper is structured into six sections. Section 1 presents the introduction. Section 2 reviews the basic concepts and related properties of rough set theory and soft set theory. In Section 3, the concept of soft rough set is introduced and its fundamental properties are examined. Section 4 focuses on the accuracy and roughness measures of a soft set, while six entropy measures of a soft set are discussed in Section 5. Finally, Section 6 provides a comparison between the proposed notion and existing approaches to the rough sets.\\

\section{Preliminaries}

Rough set theory was introduced to study uncertainty as a complementary theory to Lotfi A. Zadeh’s fuzzy set theory \cite{1} and E. C. Zeeman’s tolerance theory \cite{24}. It was introduced with the objective of enriching artificial intelligence through applications in areas such as pattern recognition, automatic classification, learning algorithms, classification theory, cluster analysis, and related fields. In defining a rough set, the notion of an indiscernibility relation plays a crucial role, as it enables the grouping of similar objects into granules. Pawlak employed an equivalence relation as the indiscernibility relation in the formulation of rough set theory \cite{6}. \\

\begin{definition}\cite{6}
    Let $X$ be a certain set called the universe and let $\mathcal{R}$ be an equivalence relation on $X$. The pair $(X,R)$ will be called an approximation space. 
Here. $R$ is known as indiscernibility relation. If $x,y\in X$ and $(x,y)\in \mathcal{R}$, then $x$ and $y$ are said to be indistinguishable in $(X,R)$.
\end{definition}

\noindent
In the aforementioned definition, the indiscernibility relation facilitates the classification of indistinguishable points (in a broader sense, objects) within the universal set  $X$. Consequently, the entire framework of rough set theory is grounded in the notion of an approximation space.\\ 

\begin{definition}\cite{6}
    Equivalence classes of the relation $\mathcal{R}$ are called elementary sets (atoms) in $(X, \mathcal{R})$. The set of all atoms in $(X,R)$ is denoted as $X/\mathcal{R}$. Every finite union of elementary sets in $(X, \mathcal{R})$ will be called a composed set in $(X, \mathcal{R})$.
\end{definition}

\begin{definition}\cite{6}
Let $K$ be a certain subset of $X$. The least composed set in $(X, \mathcal{R})$ containing $K$ will be called the  upper approximation of $K$ in $(X, \mathcal{R})$. It is denoted by $\overline{\mathrm{Apr}}(K)$, i.e., $\overline{\mathrm{Apr}}(K)
= \left\{ x : [x]_{\mathcal{R}} \in X/\mathcal{R},\ [x]_{\mathcal{R}} \cap K \neq \emptyset \right\} = \cup \left\{ [x]_{\mathcal{R}} : [x]_{\mathcal{R}} \in X/\mathcal{R},\ [x]_{\mathcal{R}} \cap K \neq \emptyset \right\}.$
\end{definition}

\begin{definition}\cite{6}
Let $K$ be a certain subset of $X$. The greatest composed set in $(X, \mathcal{R})$ contained in $K$ will be called the  lower approximation of $K$ in $(X, \mathcal{R})$. It is denoted by $\underline{\mathrm{Apr}}(K)$, i.e., $\underline{\mathrm{Apr}}(K)
= \left\{ x : [x]_{\mathcal{R}} \in X/\mathcal{R},\ [x]_{\mathcal{R}} \subseteq K  \right\} = \cup \left\{ [x]_{\mathcal{R}} : [x]_{\mathcal{R}} \in X/\mathcal{R},\ [x]_{\mathcal{R}} \subseteq K  \right\}.$
\end{definition}

\begin{definition}\cite{6}
The set $Bnd(K)= \overline{\mathrm{Apr}}(K)\setminus \underline{\mathrm{Apr}}(K)$ is called the boundary of $K$ in $(X, \mathcal{R})$.    
\end{definition}

\noindent
It can be  verified that $K = \overline{\mathrm{Apr}}(K) = \underline{\mathrm{Apr}}(K)$ if and only if $K$ is a composed set in $(X,\mathcal{R})$. In the recent literature \cite{8,9,10,11} and several other works, such a set $K$ is also referred to as an exact set or a crisp set. Consequently, it is natural to expect a mathematical expression that quantifies the accuracy of $K$ in $(X,\mathcal{R})$. To address this requirement, Pawlak \cite {6} introduced the accuracy measure of $K$.\\

\begin{definition}\cite{6}
    The accuracy of approximation of $K$ in $(X, \mathcal{R})$ is denoted by $\eta(K) = \frac{|\underline{\mathrm{Apr}}(K) |}{|\overline{\mathrm{Apr}}(K) |}$. Thus, the roughness measure of $K$ in $(X, \mathcal{R})$ is denoted by $r(K)$ and defined as $r(K)= 1 - \eta(K)$.
\end{definition}

\noindent
Similar to rough set theory, soft set theory also deals with uncertainties inherent in nature. Molodtsov \cite{12} introduced the concept of soft set theory to overcome the difficulty of selecting appropriate membership values in fuzzy set theory and to provide an alternative framework to fuzzy and rough set theories. Simon \cite{25} demonstrated that human beings are not fully rational decision-makers; instead, they exhibit bounded rationality rather than the complete rationality assumed before his work. Therefore, from many  perspectives, the introduction of soft set theory is well justified.\\

\begin{definition}\cite{12}
A pair $(S,A)$ will be called a soft set over $X$, if $S$ is a mapping from the set $A$ to the set of subsets of the set $X$, i.e., $S:A\rightarrow 2^X$. In fact, a soft set is a parametrized family of subsets. If the soft set $(S,A)$ is given, then the family $\tau(S,A)$ can be defined as  $\tau(S,A)=\{S(a): a\in A\}$. 

\end{definition}

\begin{definition} \cite{15}
 Two soft sets $(S,A)$ and $(S',A')$ given over the universal set $X$ will be called equal and we write $(S,A)=(S',A')$ if and only if $S=S'$ and $A=A'$.
\end{definition}

\begin{definition} \cite{15}
Two soft sets $(S,A)$ and $(S',A')$ given over the universal set $X$ will be called equivalent and written as $(S,A)\cong (S',A')$ if and only if $\tau(S,A)=\tau(S',A')$.

\end{definition}

\noindent
 From Definition 2.7,  it is possible that for some $a\in A$, we may obtain $S(a)= \emptyset$.  Thus, throughout this paper, we  use $\tau'(S,A)=\{\, S(a): S(a)\neq \emptyset \quad \forall\, a\in A \,\}$ as a conventional notion to indicate the aforementioned  special case if it may exist for a soft set $(S,A)$. Moreover, for two soft sets $(S,A)$ and $(F,A)$ defined over $X$, we  use the conventional notion $\tau'(S,A)=\tau'(F,A)$ to indicate that $S(a)=F(a) (\neq \emptyset) \quad \forall a\in A$, in conjunction with the previously mentioned concept.  \\

\noindent
 According to Molodtsov \cite{15}, each soft set is a representative of its equivalence class. The distinction between equivalent soft sets lies solely in the naming of their subsets, including cases where multiple names are assigned to the same subset. When the family $\tau(S, A)= \emptyset$, then the equivalence class generated by such a soft set will be called an empty soft set \cite{15} and denoted by the symbol $(-, \emptyset)$, since it corresponds to an empty set of parameters. One should not confuse an empty soft set $(-, \emptyset)$ with the equivalence class generated by the soft set $(S, A)$ for which the equality $\tau (S, A) = \{\emptyset\}$ is satisfied. The latter soft set is defined on a non-empty parameter set $A$, and for every $a \in A$, we have $S(a)=\emptyset$.
 The equivalence class generated by such a soft set will be called a null soft set \cite{15} and denoted by the symbol $(\emptyset, -)$. Similarly, if the equivalence class is generated by the soft set $(S,A)$, for which the equality $\tau(S,A)=\{X\}$ is satisfied, then we shall denote it by the symbol $(X, -)$. It means that for each parameter $a \in A$, the condition $S(a) = X$ is satisfied.\\

\noindent
It is  important to note that Molodtsov \cite{15} used the symbols $\subseteq$ and $\supseteq$ in two distinct contexts without causing any conceptual overlap. In the case of internal approximations between two soft sets, the symbol $\subseteq$ was employed. Conversely, for external approximations, the symbol $\supseteq$ was used. However, when attribute approximations are defined via subsets of a universal set in the context of a soft set, the symbols $\subseteq$ and $\supseteq$ are employed according to their classical set-theoretic meanings. Following Molodtsov \cite{13,15,16}, we adopt the same conventions, as they do not impact our study, given that the notions of subsets or supersets do not inherently exist within soft set theory.
  \\

\begin{definition} \cite{15}
A soft set $(S,A)$ internally approximates a soft set $(F,D)$ denoted by $(S,A)\subseteq (F,D)$ if for any $d\in D$ such that $F(d)\neq \phi$, there exists $a\in A$ for which $\phi \neq S(a)\subseteq F(d)$.
\end{definition}

\begin{definition} \cite{15}
A soft set $(S,A)$ externally approximates a soft set $(F,D)$ denoted by $(S,A)\supseteq (F,D)$ if for any $d\in D$ such that $F(d)\neq X$, there exists $a\in A$ for which $X \neq S(a)\supseteq F(d)$.
\end{definition}

\begin{definition} \cite{15}
If $(S,A)\subseteq (F,D)$ but the relation $(F,D)\subseteq (S,A)$ has no place, then we say that the soft set $(S,A)$ internally strictly approximates the soft set $(F,D)$, denoted by $(S,A)\subset (F,D)$.
\end{definition}

\begin{definition} \cite{15}
If $(S,A)\supseteq (F,D)$ but the relation $(F,D)\supseteq (S,A)$ has no place, then we say that the soft set $(S,A)$ externally strictly approximates the soft set $(F,D)$, denoted by $(S,A)\supset (F,D)$.
\end{definition}

\begin{definition} \cite{15}
The soft set $(S,A)$ is internally equivalent to the soft set $(F,D)$, denoted by $(S,A) \begin{matrix}
\subset \\
\approx
\end{matrix} (F,D)$ if $(S,A)\subseteq (F,D)$ and $(F,D)\subseteq (S,A)$.
\end{definition}

\begin{definition} \cite{15}
The soft set $(S,A)$ is externally equivalent to the soft set $(F,D)$, denoted by $(S,A) \begin{matrix}
\supset \\
\approx
\end{matrix} (F,D)$ if $(S,A)\supseteq (F,D)$ and $(F,D)\supseteq (S,A)$.
\end{definition}

\noindent
Intuitively, the internal approximation described in Definition 2.10 implies that a soft set can expand within another soft set without placing any restrictions on the attributes. On the other hand, the external approximation presented in Definition 2.11 reflects the general idea of contraction or squeezing. Expansion and contraction are important aspects in image segmentation and pattern recognition \cite{26}. Therefore, the proposed theoretical frameworks can support simultaneous computational segmentation and recognition in these areas, offering an approach that differs from existing methods.\\

\noindent
Molodtsov \cite{12,13} urged that, in defining binary operations on two soft sets $(S,A)$ and $(F,D)$ over a universal set $X$, the resulting soft set should naturally be equipped with the attribute set $A \times D$. For further details on conceptual constraints and related theoretical aspects of soft set theory, the reader may consult \cite{14,15, 16, 27}. However, these fundamental philosophical standpoints have not been consistently respected in much of the existing literature. Although Molodtsov \cite{12} did not explicitly state the motivation behind this specific construction of the attribute set, one may nevertheless draw technological inspiration from this underlying philosophical idea.
During the process of matching two images to assess similarity or dissimilarity, human perception relies on comparing attributes and their corresponding attribute-based approximation subsets within the images. Through this comparative evaluation, similarities or dissimilarities are identified \cite{28, 29}. Thus, from a philosophical standpoint, it is justified to consider the new attribute set as the Cartesian product of the attribute sets of the two soft sets when defining binary operations among them.  In that case, each ordered pair may represent the ordered pair of pixels of two images, considering the images as soft sets \cite{44,45,46,47,48}. Hence, the following operations are both theoretically grounded \cite{12,13} and philosophically justifiable \cite{28,29}.\\

\begin{definition} \cite{15}
The unary operation complement of $(S,A)$,  $C(S,A)=(W,A)$ is defined as follows: the set of parameters remains the same and the mapping is given by $W(a)=X\setminus S(a)$, for any $a\in A$.
\end{definition}

\begin{definition} \cite{15}
The binary operation union $(S,A)\cup (F,D)=(H,A\times D)$ for a pair of soft sets $(S,A)$ and $(F,D)$ given over a universal set $X$ is defined as follows: the parameter set is chosen equal to the direct product of the  parameter sets, i.e., equal to $A\times D$, and the mapping is given by the formula $H(a,d)=S(a)\cup F(d)$, $(a,d)\in A\times D$.
\end{definition}

\begin{definition} \cite{15}
The binary operation intersection $(S,A)\cap (F,D)= (W,A\times D)$  for a pair of soft sets $(S,A)$ and $(F,D)$ given over a universal set $X$ is defined as follows: the parameter set is chosen equal to the direct product of the parameter sets, i.e., equal to $A\times D$, and the mapping is given by the formula $W(a,d)=S(a)\cap F(d)$, $(a,d)\in A\times D$.
\end{definition}

\begin{definition}\cite{14}
The binary operation product $(S,A)\times (F,D)=(R, A\times D)$ for a pair of soft sets $(S,A)$ and $(F,D)$ given over a universal set $X$ is defined as follows: the parameter set is chosen equal to the direct product of the parameter sets, i.e., equal to $A\times D$ and the mapping is given by the formula $R(a,d)=S(a) \times F(d), (a,d)\in A\times D$.
\end{definition}

\noindent

\section{Roughness of a soft set}

A soft set is associated with a collection of attributes together with the corresponding approximation sets of those attributes \cite{12,13,15}. Hence, assessing the roughness of a soft set in relation to its attributes is of fundamental importance. The development of this new structure enables the simultaneous distinction of multiple objects within an environment, regardless of issues arising from their mutual overlap. Within this new theoretical framework, each attribute is associated with an attribute-approximation subset that possesses both lower and upper approximations at the individual level.\\ 

\noindent
In \cite{20,21,22,23} and related works, different notions of roughness in soft set theory have been developed using subsets of the attribute sets.  Their methodologies depart from the original philosophical foundation and formal structure of soft set theory as introduced by Molodtsov \cite{12, 13,15,18,19}. Moreover, the operations adopted in these works are not aligned with Molodtsov’s original soft set operations \cite{12, 13,15,18,19}. Therefore, there is a compelling need to develop a rigorous and coherent theoretical framework for roughness structures in soft set theory that is fully consistent with and comparable to Molodtsov’s foundational principles. Hence, we establish the following definitions and results:\\

\begin{definition} Let $(X, \mathcal{R})$ be an approximation space and $(S, A)$ be a soft set defined over $X$. Then, lower soft approximation of $(S, A)$ is denoted as $L(S,A)$  and defined as $L(S,A)=\{(a_i, \cup_j B_j) : a_i\in A, B_j \in X/\mathcal{R}, j\in \Delta, \Delta \ \text{ is an index set and} \ B_j \subseteq S(a_i)\}.$
\end{definition}

\begin{definition}
  Let $(X, \mathcal{R})$ be an approximation space and $(S, A)$ be a soft set defined over $X$. Then, upper soft approximation of $(S, A)$ is denoted as $U(S,A)$  and defined as 
  $U(S,A)=\{(a_i, \cup_k C_k) : a_i\in A, C_k \in X/\mathcal{R}, k\in \Delta, \Delta \ \text{ is an index set and} \ C_k \cap  S(a_i) \neq \emptyset \}.$
 
\end{definition}

\noindent
The 4-tuple $(X,R, L(S,A), U(S,A))$ is called a soft rough set of the soft set $(S,A)$. Eventually, we can also write it as $(L(S,A), U(S,A))$  in short for a fixed universal set $X$ and for a fixed equivalence relation $\mathcal{R}$, without any confusion. For the sake of our easiness, throughout this paper wherever necessary,  we consider the notions $\underset{\sim}{S}(a_i)$ and $\overset{\sim}{S}(a_i)$ to denote $a_i$-approximation sets of $L(S,A)$ and $U(S,A)$, respectively, where $a_i\in A$. Moreover, $\underset{\sim}{S}(a_i)$ and $\overset{\sim}{S}(a_i)$ will be called the lower soft  and upper soft approximations of subsets $S{(a_i)}$ $\forall i\in A$, respectively.

\begin{definition}
  Let $(X, \mathcal{R})$ be an approximation space,  $(S, A)$ be a soft set defined over $X$ and $a_i\in A$. Then, boundary of an $a_i$-approximate element   $S(a_i)$ of $(S, A)$ is denoted by $Bd (S(a_i))$ and  defined as $Bd (S(a_i))= \overset{\sim}{S}(a_i) \setminus \underset{\sim}{S}(a_i)$.
\end{definition}

\noindent

\begin{definition}
\noindent
  Let $(X, \mathcal{R})$ be an approximation space and  $(S, A)$ be a soft set defined over $X$. The total boundary of  $(S, A)$ is denoted by $Bd(S, A)$ and defined as $Bd(S,A)= \cup_{a_i\in A} Bd (S(a_i)).$
\end{definition}

\noindent
As shown in Figure 1, a soft set $(S, A)$ is defined over $X$ and $(X, \mathcal{R})$ forms an approximation space. 
The grids depict the equivalence classes of $X$ induced by  the equivalence relation $\mathcal{R}$. We assume that the soft set $(S,A)$ contains three $\epsilon$-approximation subsets, where $\epsilon\in A$. Accordingly, the three $\epsilon$-approximation subsets are illustrated as the colored regions-red, blue, and green. The upper soft approximations of these subsets are bounded by black colored straight lines. In this two dimensional representation, the upper soft approximations appear to overlap, but the figure would be clearer if shown using layered structures.

\begin{figure}[H]
    \centering
    \includegraphics[width= 6 cm, height= 6 cm]{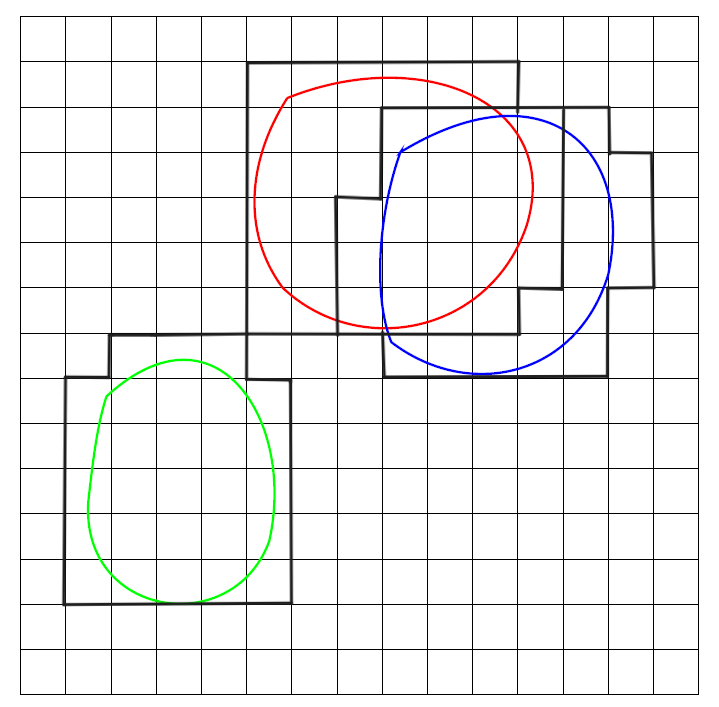} 
    \caption{ Upper soft approximation of  a soft set $(S, A)$ determined  over an approximation space $(X, \mathcal{R})$.}
\end{figure}

\begin{figure}[H]
    \centering
    \includegraphics[width= 6 cm, height= 6 cm]{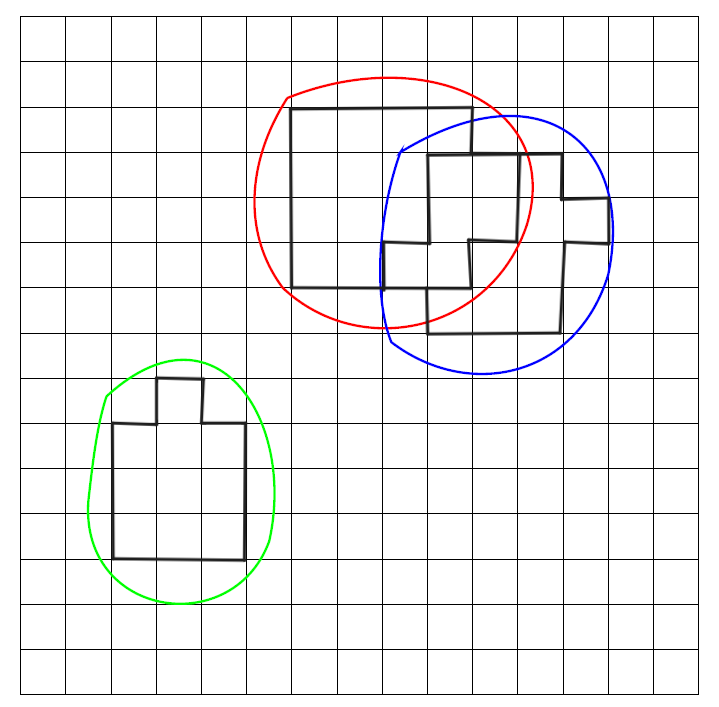} 
    \caption{ Lower soft approximation of  a soft set $(S, A)$ determined  over an approximation space $(X, \mathcal{R})$.}
\end{figure}
\noindent
Similarly, in Figure 2, the lower soft approximations of these subsets are bounded by  black colored straight lines. The clear distinctions between upper soft approximations and lower soft approximations of these subsets help to identify overlaps between the images. This feature differentiates the concepts introduced in this paper from those in rough set theory and its hybrid variants, as well as from the methods defined in this section. \\

\begin{figure}[H]
    \centering
    \includegraphics[width= 6 cm, height= 6 cm]{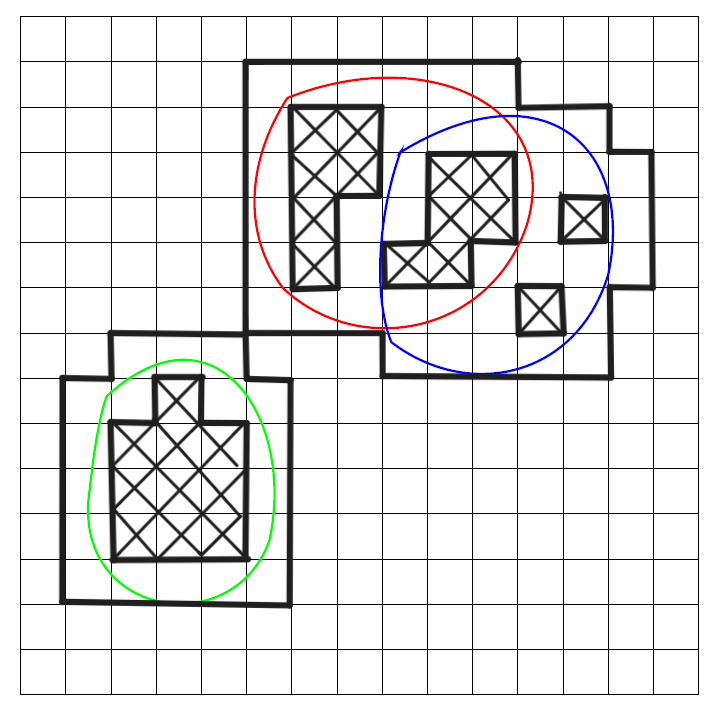} 
    \caption{ Boundary of a soft set $(S, A)$ determined  over an approximation space $(X, \mathcal{R})$.}
\end{figure}

\noindent
In Figure 3, the boundary of the soft set $(S,A)$ is determined using Definition 3.3. The boundary $Bd(S,A)$ is represented by the region enclosed between the polygons shown with bold black straight lines (excluding the shaded regions). Now, we discuss the following example to discuss our aforementioned notions:\\

\begin{example}
    Let us consider a universal set $X=\{a,b,c,d\}$. We define an equivalence relation $\mathcal{R}$ on $X$ as $\mathcal{R} = \{(a,a), (b,b), (c,c), (d,d), (a,b), (b,a), (c,d), (d,c)\}$. Then, we have two distinct equivalence classes. The equivalence classes of $a$ and $b$ are identical. Similarly, the equivalence classes of $c$ and $d$ are identical. We denote the equivalence classes of $a, b, c$ and $d$ by $[a], [b], [c]$, and $[d]$, respectively.
    So, $[a]=\{a,b\}= [b]$, and $[c]=\{c,d\}=[d]$. Hence,      $X/ \mathcal{R} =\{[a], [c]\}$.\\

\noindent
Now, let $(S, A)$ be a soft set defined over $X$ with $A=\{a_1, a_2, a_3\}$, where $S(a_1)= \{a, b, d\}$, $S(a_2)= \{b, c, d\}$, and $S(a_3)= \{ c\}$. We get $L(S, A)= \{(a_1, \{a,b \}), (a_2, \{ c,d\}), (a_3, \emptyset)\}$
, and $U(S, A)= \{(a_1, \{a,b,c,d\}), (a_2, \{a,b,c,d \}), (a_3, \{c,d\} )\}$. Hence,  $Bd((S(a_1))=\{c,d\}$, $Bd((S(a_2))=\{a,b\}$,  and  $Bd((S(a_3))=\{c,d\}$. So, $Bd(S,A)=X$.
\end{example}

\begin{definition}
    Let $(X, \mathcal{R})$ be an approximation space and $(S, A)$ be a soft set defined over $X$. If $L(S,A) = U(S, A)$, then $(S,A)$ is said to be an exact  soft set. In this case, it is obvious that $\tau(L(S,A)) = \tau(U(S,A))$.
\end{definition}

\begin{definition}
    Let $(X, \mathcal{R})$ be an approximation space,  $(S, A)$ and $(F,D)$ be two soft sets defined over $X$. Then, the soft sets $(S, A)$ and $(F, D)$ are said to be  roughly soft equal if $\tau(L(S,A))= \tau(L(F, D))$ and  $\tau(U(S,A))= \tau(U(F, D))$. We denote it by $(S,A)\asymp (F, D)$
    
\end{definition}
\begin{example} Let us consider Example 3.1. Here, we get $\tau(L(S,A))=\{\{a,b\}, \{c,d\}, \emptyset\}$ and $\tau(U(S,A))=\{\{a,b, c,d\}, \{c,d\}\}$. Now, we consider another soft set $(F, D)$, where $D=\{k_1, k_2, k_3, k_4\}$ and $F(k_1)= \{ a, b, c\}$, $F(k_2)= \{ b,d \}$, $F(k_3)= \{ c,d \}$, and $F(k_4)= \{  a, d  \}$.
Then, we can find $L(F,D)= \{ (k_1, \{ a,b   \}),
(k_2, \emptyset), (k_3, \{ c,d \}), (k_4, \emptyset)\}$ and $U(F,D)= \{ (k_1, \{ a,b,c,d   \}), (k_2, \{a,b,c,d\}),\\ (k_3, \{ c,d\}), (k_4, \{a,b,c,d\} )\}$. Then, $\tau(L(F,D))=\{\{a,b\}, \{c,d\}, \emptyset\}$ and $\tau(U(F,D)=\{\{a,b, c,d\}, \{c,d\}\}$. Hence, 
$(S, A)$ and $(F,D)$ are roughly equal.
\end{example}

\noindent
The aim of Definition 3.6 is to show that two soft sets with different attribute sets can still produce the same $\tau$-class for both lower and upper soft approximations of the soft sets. Because the existing notions of soft subsets of soft sets in the soft set theory literature undermine the original theoretical framework proposed by Molodtsov \cite{15,16}, it is essential to build the theory using the internal and external approximations given in Definitions 2.10 and 2.11. Accordingly, the concepts of soft subsets of soft sets are not used in the proofs of all subsequent results presented in this work. This step also preserves one of Molodtsov’s fundamental conceptual barriers in soft set theory \cite{13, 15, 16}, which is not maintained in \cite{20, 21, 22, 23}.\\

\begin{theorem}
Let $(X, \mathcal{R})$ be an approximation space,  $(S, A)$  and $(F, D)$  be two  soft sets defined over $X$. Then, the following results hold: \\

(i) if $(S,A) \subseteq (F, D)$, then $L(S,A) \subseteq L(F, D)$  and $U(S,A) \subseteq U(F, D)$,\\

(ii) if $(S,A) \subset (F, D)$, then $L(S,A) \subset L(F, D)$ and $U(S,A) \subset U(F, D)$.\\

\end{theorem}

\begin{proof} (i) We assume that $(S,A) \subseteq (F, D)$ holds.  Now, we first prove that $L(S,A) \subseteq L(F, D)$. If possible, we consider $L(S,A) \nsubseteq L(F, D)$. Then, by Definition 2.10, for any $d\in D$ such that $\underset{\sim}{F}(d)\neq \emptyset$, there does not exist any $a\in A$ such that $\emptyset \neq \underset{\sim}{S}(a) \subseteq \underset{\sim}{F}(d)$ holds.  As $\underset{\sim}{F}(d)\neq \emptyset$, thus $F(d)\neq \emptyset$. Now, it is easy to verify that if $ \underset{\sim}{S}(a) \subseteq \underset{\sim}{F}(d)$, then  we have ${S}(a) \subseteq {F}(d)$.   Thus, we obtain that for any $d\in D$ such that ${F}(d)\neq \emptyset$, there does not exist any $a\in A$ such that $\emptyset \neq {S}(a) \subseteq {F}(d)$ holds. It gives  $(S,A) \nsubseteq (F, D)$, which is a contradiction. So, $L(S,A) \subseteq L(F, D)$. \\
  
\noindent
Now, we prove that $U(S,A)\subseteq U(F, D)$. Since $(S,A)\subseteq (F, D)$, for any $d\in D$ such that ${F}(d)\neq \emptyset$, there exists  $a\in A$ such that $\emptyset \neq {S}(a) \subseteq {F}(d)$ holds.  As ${F}(d)\neq \emptyset$, thus $\overset{\sim}F(d)\neq \emptyset$. Now, it is easy to verify that if $ {S}(a) \subseteq {F}(d)$, then  we have $\overset{\sim}{S}(a) \subseteq \overset{\sim}{F}(d)$. Thus, for any $d\in D$ such that $\overset{\sim}{F}(d)\neq \emptyset$, there exists  $a\in A$ such that $\emptyset \neq \overset{\sim} {S}(a) \subseteq \overset{\sim}{F}(d)$ holds.  Hence, we $U(S,A)\subseteq U(F,D)$. \\ 

\noindent
(ii)  Since $(S,A) \subset (F,D)$,  due to Definition 2.12,  we get $(S,A)\subseteq (F,D)$ but not the relation $(F,D)\subseteq (S,A)$. So, by (i),  we get $L(S,A)\subseteq L(F,D)$  but not $L(F,D)\subseteq L(S,A)$. Hence, $L(S,A) \subset L(F,D)$. Similarly, we can prove that $U(S,A) \subset U(F,D)$. 
\end{proof}

\noindent
The aforementioned theorem establishes that if one soft set internally approximates (resp.  internally strictly approximates) another soft set on the same approximation space, then the lower and upper soft approximations of the first soft set also internally approximate ( resp. internally strictly approximates ) the corresponding lower and upper soft approximations of the second soft set. This result can be interpreted in terms of image analysis: it allows one to expand the regions (or parts) of one image so that they lie within the corresponding regions (or parts) of another image, without needing to consider specific attribute information. Furthermore, it also facilitates the consideration of internal deformation. Conversely, the following theorem states that if one soft set externally approximates another soft set within the same approximation space, then the lower soft approximation of the first soft set likewise externally approximates the corresponding lower soft approximation of the second. It indicates approximations from the external side. \\

\begin{theorem}
Let $(X, \mathcal{R})$ be an approximation space,  $(S, A)$  and $(F, D)$  be two  soft sets defined over $X$. Then, the following results hold:  \\

(i) if $(S,A) \supseteq (F, D)$, then   $L(S,A) \supseteq L(F, D)$, \\

(ii) if $(S,A) \supset (F, D)$, then  $L(S,A) \supset L(F, D)$.\\

\end{theorem}

\begin{proof}
The proofs of (i) and (ii) can be easily obtained. Hence, they are skipped. 
\end{proof}
\noindent
Unlike in the case of internal approximation, the relation $U(S,A)\supseteq U(F,D)$ does not necessarily hold for external approximation. Hence, we have the following remark and Example 3.3 justifies it.\\

\begin{remark}
  Let $(X, \mathcal{R})$ be an approximation space,  $(S, A)$  and $(F, D)$  be two  soft sets defined over $X$.  If $(S,A) \supseteq (F, D)$, then    $U(S,A) \supseteq U(F, D)$ is not true in general. Thus, if $(S,A) \supset (F, D)$, then    $U(S,A) \supset U(F, D)$ is also not true in general.\\
\end{remark}

\begin{example}
    Let $X= \{a,b,c,d,e\}$. Consider an equivalence relation  $R=\{(a,a),(b,b),(c,c),(d,d),\\(e,e),(a,b),(b,a),(d,e),(e,d)\}$. Then, the equivalence classes are  $[a]=[b]=\{a,b\},\; [c]=\{c\},\; [d]=[e]=\{d,e\}$. Thus, $X/\mathcal{R}= \{[a], [c], [d]\}$.\\

\noindent
We consider two soft sets $(S,A)$ and $(F,D)$, respectively, where $A=\{a_1, a_2, a_3\}$ and $D=\{d_1, d_2\}$. Now, we define $S(a_1)=\{a,b,c,d\}, S(a_2)=\{c,e,b\}, S(a_3)= \{b\}, F(d_1)=\{a,b,d\}$ and $ F(d_2)=\{c,e\}$. Then, it can be seen that $(S,A)\supseteq (F,D)$, but $U(S,A)\supseteq U(F,D)$ is not true.
\end{example}

\noindent
Analogous to internal approximation, the following two theorems examine the relationship among internally strict approximations, lower soft approximations, and upper soft approximations. \\

\begin{theorem}
 Let $(X, \mathcal{R})$ be an approximation space,  $(S, A)$  and $(F, D)$  be two  soft sets defined on $X$. If $(S,A) \begin{matrix} \subset \\
\approx
\end{matrix} (F,D)$,  then $L(S,A) \begin{matrix} 
\subset \\
\approx
\end{matrix} L(F,D)$ and  $U(S,A) \begin{matrix} 
\subset \\
\approx
\end{matrix} U(F,D)$.
\end{theorem}

\begin{proof}
(i) Since $(S,A) \begin{matrix}
\subset \\
\approx
\end{matrix} (F,D)$, Definition 2.14 yields $(S,A)\subseteq (F,D)$ and $(F,D)\subseteq (S,A)$. Hence, by Theorem 3.1, we can say that $L(S,A)\subseteq L(F,D)$ and $L(F,D)\subseteq L(S,A)$. Thus, $L(S,A) \begin{matrix}
\subset \\
\approx
\end{matrix} L(F,D)$. Similarly, we can obtain the proof of  the remaining part. 
\end{proof}

\noindent
In the preceding section, we discussed the notions $(\emptyset, -)$ and $(X, -)$, which are the equivalence classes of soft sets $(S, A)$ and $(F, D)$, respectively, such that $\tau(S, A)=\{\emptyset\}$ and $\tau(F, D)=\{X\}$. Thus, we obtain the following theorem:

\begin{theorem}
    Let $(X, \mathcal{R})$ be an approximation space,  $(S, A)$  and $(F, D)$  be two  soft sets defined over $X$. Then, the following results hold: \\

    (i) $L(S, A) = (S,A)$, and  $U(S, A) = (S,A)$ if $(S,A)\in (\emptyset, -)$.

    (ii) $L(F, D ) = (F,D)$, and  $U(F, D ) = (F, D )$ if $(F,D)\in (X, -)$.

    (iii) $L(L(S,A)) = L(S,A)$, and $U(U(S,A)) = U(S,A)$,

    (iv) $L(U(S,A)) = U(S,A)$, and $U(L(S,A)) = L(S,A)$.

\end{theorem}

\begin{proof}
    (i) Let $(S,A)\in (\emptyset, -)$. Then,  according to Molodtsov \cite{15, 16}, $\tau(S,A)=\{ \emptyset \}$.
So, it means that $S(a_i)=\emptyset$ $\forall a_i\in A$. Thus, $\underset{\sim}{S}(a_i)= \emptyset $ $\forall i\in A$. Hence, 
    $L(S, A) = (S,A)$.\\

Again, $S(a_i)=\emptyset$ $\forall a_i\in A$. Thus, $\overset{\sim}{S}(a_i)= \emptyset $ $\forall a_i\in A$. Hence,  $U(S, A) = (S,A)$.\\

(ii) Let $(F,D)\in (X, -)$. Then,  $\tau(F,D)=\{ X\}$. So, it means that $F(d_i)=X$ $\forall d_i\in D$. Thus, $\underset{\sim}{F}(d_i)= X $ $\forall d_i\in D$. Hence,  $L(F, D) = (F,D)$.\\

Again $F(d_i)=X$ $\forall d_i\in D$.  Hence, $\overset{\sim}{F}(d_i)= X $ $\forall d_i\in D$. Hence,  $U(F, D) = (F, D)$. \\

(iii) By Definition 3.1, we have $L(S,A)=\{(a_i, \cup_j B_j) : a_i\in A, B_j \in X/\mathcal{R}, j\in \Delta, \Delta \ \text{ is an index set and}\\ \ B_j \subseteq S(a_i)\}$. Here, $\underset{\sim}{S}(a_i)= \cup_j B_j $. Now, for the sake of easiness,  let us denote the soft set $L(S,A)$ as $(R,A)$. Then, $(R,A)=\{(a_i, \cup_j B_j) : a_i\in A, B_j \in X/\mathcal{R}, j\in \Delta, \Delta \ \text{ is an index set and}\ B_j \subseteq S(a_i)\}$, where $R(a_i)= \cup_j B_j$. Since  $B_j \subseteq {R}(a_i)$ $\forall j\in \Delta$, we have  $\underset{\sim}{R}(a_i)=R(a_i)$. Thus,  $L(R, A)= (R,A)$, i.e., $L(L(S,A))=(S,A)$.\\

\noindent
Using Definition 3.2, it can be easily seen that $U(U(S,A))= U(S,A)$.\\ 

(iv) We know, $U(S,A)=\{(a_i, \cup_k C_k) : a_i\in A, C_k \in X/\mathcal{R}, k\in \Delta, \Delta \ \text{ is an index set and} \ C_k \cap  S(a_i) \neq \emptyset \}.$
Here, $\overset{\sim}{S}(a_i)= \cup_k C_k$. Then, $C_k\subseteq \overset{\sim}{S}(a_i) $.
Let us denote the soft set $U(S,A)$ as $(M,A)$. Then, $L(M,A)= \{(a_i, \cup_k C_k) : a_i\in A, C_k \in X/\mathcal{R}, k\in \Delta, \Delta \ \text{ is an index set and} \ C_k \subseteq \overset{\sim} {S}(a_i)  \}$, since $\underset{\sim} {M}(a_i) = \cup_kC_k$. Thus, it is easy to find that
$L(U(S,A))= U(S,A)$. We can easily prove that $U(L(S,A))= L(S,A)$. Thus, the proof can be skipped.    
\end{proof}

\noindent
In line with the preceding result, it is essential to investigate how the lower and upper soft approximations influence the binary operations, namely soft union $(\cup)$ and soft intersection $(\cap)$, within the frameworks of internal and external approximations. Such influences are particularly valuable in pattern recognition and image processing, as they facilitate the component-wise combination or separation of two images, followed by the expansion or contraction of their respective regions. Accordingly, the following two theorems are established:\\

\begin{theorem}
Let $(X, \mathcal{R})$ be an approximation space,  $(S, A)$  and $(F, D)$  be two  soft sets defined over $X$. If $\underset{\sim}{S}(a_i) \cup \underset{\sim}{F}(d_j) \neq \emptyset$, and  $\underset{\sim}{S}(a_i) \cap \underset{\sim}{F}(d_j) \neq \emptyset$ $\forall a_i\in A, \forall d_j\in D$, then the following results hold: \\

(i) $L(S,A)\cup L(F,D)\subseteq L((S,A)\cup (F,D))$,\\

(ii) $ L((S,A)\cap (F,D)) \subseteq L(S,A)\cap L(F,D)$,\\

(iii) $U(S,A)\cup U(F,D)\subseteq U((S,A)\cup (F,D))$,\\

(iv) $U((S,A)\cap (F,D))\subseteq U(S,A)\cap U(F,D)$.
\end{theorem}
\begin{proof}
(i)  By Definition 2.17, we have $L(S,A)\cup L(F,D)= (H, A\times D)$, where $H(a_i,d_j)= \underset{\sim}{S}(a_i) \cup \underset{\sim}{F}(d_j), a_i\in A, d_j\in D, i\in \Delta_1, j\in \Delta_2$ and $\Delta_1, \Delta_2$ are two  index sets. Now, we consider $(S,A) \cup (F,D) = (R, A\times D)$. Then, $L((S, A)\cup (F, D))= L(R, A\times D) = \{((a_i, d_j),  \cup_r Q_r) : a_i\in A, d_j\in D,  Q_r \in X/\mathcal{R}, i\in \Delta_1, j\in \Delta_2,  r\in \Delta; \Delta_1, \Delta_2, \Delta  \ \text{ are  index sets and} \ Q_r \subseteq S(a_i)\cup F(d_j)\}$. Here, $\underset{\sim}{R}(a_i, d_j)= \cup_r Q_r$. \\

\noindent
Let us assume that for any $(a_i,d_j)\in A\times D$, we get $\underset{\sim}{R}(a_i, d_j)\neq \emptyset$. Since, $\underset{\sim}{S}(a_i) \cup \underset{\sim}{F}(d_j) \neq \emptyset$,  $\forall a_i\in A, \forall d_j\in D$, then either $\underset{\sim}{S}(a_i) \neq \emptyset$ or $\underset{\sim}{F}(d_j) \neq \emptyset$. For all cases, 
$\underset{\sim}{S}(a_i) \subseteq \underset{\sim}{R}(a_i,d_j)$ and  $\underset{\sim}{F}(d_j) \subseteq \underset{\sim}{R}(a_i,d_j)$. Thus, $ \emptyset\neq \underset{\sim}{S}(a_i) \cup \underset{\sim}{F}(d_j) \subseteq \underset{\sim}{R}(a_i,d_j)$. Thus,  for any $(a_i,d_j)\in A\times D$, such that $\underset{\sim}{R}(a_i, d_j)\neq \emptyset$, we get the same $(a_i, d_j)\in A\times D$ for which $ \emptyset\neq \underset{\sim}{S}(a_i) \cup \underset{\sim}{F}(d_j) \subseteq \underset{\sim}{R}(a_i,d_j)$ holds. Hence, $L(S,A)\cup L(F,D)\subseteq L((S,A)\cup (F,D))$,\\ 

\noindent
(ii) By Definition 2.18, we have $L(S,A)\cap L(F,D)= (H', A\times D)$, where $H'(a_i,d_j)= \underset{\sim}{S}(a_i) \cap \underset{\sim}{F}(d_j), a_i\in A, d_j\in D, i\in \Delta_1, j\in \Delta_2$ and $\Delta_1, \Delta_2$ are two  index sets. Now, we consider $(S,A) \cap (F,D) = (R', A\times D)$. Then, $L((S, A)\cap (F, D))= L(R', A\times D) = \{((a_i, d_j),  \cup_r Q_r) : a_i\in A, d_j\in D,  Q_r \in X/\mathcal{R}, i\in \Delta_1, j\in \Delta_2,  r\in \Delta; \Delta_1, \Delta_2, \Delta  \ \text{ are  index sets and} \ Q_r \subseteq S(a_i)\cap F(d_j)\}$. Here, $\underset{\sim}{R'}(a_i, d_j)= \cup_r Q_r$. \\

\noindent
Let us assume that for any $(a_i,d_j)\in A\times D$, we get $H'(a_i,d_j)= \underset{\sim}{S}(a_i) \cap \underset{\sim}{F}(d_j)\neq \emptyset$. Since, $\underset{\sim}{S}(a_i) \cap \underset{\sim}{F}(d_j) \neq \emptyset$,  $\forall a_i\in A, \forall d_j\in D$, then $\underset{\sim}{S}(a_i) \neq \emptyset$ and $\underset{\sim}{F}(d_j) \neq \emptyset$. Thus, $S(a_i)\neq \emptyset$ and $F(d_j)\neq \emptyset$.
Then, obviously $\underset{\sim}{R'}(a_i, d_j)= \cup_r Q_r\neq \emptyset$ and $\emptyset\neq \underset{\sim}{R'}(a_i, d_j) \subseteq \underset{\sim}{S}(a_i) \cap \underset{\sim}{F}(d_j)$.  Thus,  for any $(a_i,d_j)\in A\times D$, such that $ \emptyset\neq \underset{\sim}{S}(a_i) \cap \underset{\sim}{F}(d_j) $,  we get the same $(a_i, d_j)\in A\times D$ for which $ \emptyset\neq \underset{\sim}{R'}(a_i,d_j) \subseteq  \underset{\sim}{S}(a_i) \cap \underset{\sim}{F}(d_j)$ holds. Hence, $L((S,A)\cap (F,D))\subseteq L(S,A)\cap L(F,D)$,\\ 

\noindent
Similarly, we can prove (iii) and (iv). Thus, we skip the remaining proofs. 
   
\end{proof}

\begin{theorem}
Let $(X, \mathcal{R})$ be an approximation space,  $(S, A)$  and $(F, D)$  be two  soft sets defined over $X$. If $\overset{\sim}{S}(a_i) \cup \overset{\sim}{F}(d_j) \neq X$ and $\overset{\sim}{S}(a_i) \cap \overset{\sim}{F}(d_j) \neq X$ $\forall a_i\in A, \forall d_j\in D$, then the following results hold: \\

(i) $L((S,A)\cup (F,D))\supseteq L(S,A)\cup L(F,D)$,\\

(ii) $ L(S,A)\cap L(F,D) \supseteq L((S,A)\cap (F,D))$,\\

(iii) $U((S,A)\cup (F,D))\supseteq U(S,A)\cup U (F,D)$,\\

(iv) $ U(S,A)\cap U(F,D) \supseteq U((S,A)\cap (F,D))$.
\end{theorem}

\begin{proof}
    (i) By Definition 2.17, we have $L(S,A)\cup L(F,D)= (H, A\times D)$, where $H(a_i,d_j)= \underset{\sim}{S}(a_i) \cup \underset{\sim}{F}(d_j), a_i\in A, d_j\in D, i\in \Delta_1, j\in \Delta_2$ and $\Delta_1, \Delta_2$ are two  index sets. Now, we consider $(S,A) \cup (F,D) = (R, A\times D)$. Then, $L((S, A)\cup (F, D))= L(R, A\times D) = \{((a_i, d_j),  \cup_r Q_r) : a_i\in A, d_j\in D,  Q_r \in X/\mathcal{R}, i\in \Delta_1, j\in \Delta_2,  r\in \Delta; \Delta_1, \Delta_2, \Delta  \ \text{ are  index sets and} \ Q_r \subseteq S(a_i)\cup F(d_j)\}$. Here, $\underset{\sim}{R}(a_i, d_j)= \cup_r Q_r$. \\

\noindent
Let us assume that for any $(a_i,d_j)\in A\times D$, we get $H(a_i, d_j)= \underset{\sim}{S}(a_i) \cup \underset{\sim}{F}(d_j)\neq X$. Since, $\underset{\sim}{S}(a_i) \cup \underset{\sim}{F}(d_j) \neq X$,  $\forall a_i\in A, \forall d_j\in D$, then neither $\underset{\sim}{S}(a_i) \neq X$ nor $\underset{\sim}{F}(d_j) \neq X$. For all cases, $\underset{\sim}{S}(a_i) \subseteq \underset{\sim}{R}(a_i,d_j)$ and  $\underset{\sim}{F}(d_j) \subseteq \underset{\sim}{R}(a_i,d_j)$. Thus, $ X\neq  \underset{\sim}{R}(a_i,d_j) \supseteq \underset{\sim}{S}(a_i) \cup \underset{\sim}{F}(d_j) $.  Thus,  for any $(a_i,d_j)\in A\times D$, such that $\underset{\sim}{S}(a_i) \cup \underset{\sim}{F}(d_j)\neq X$, we get the same $(a_i, d_j)\in A\times D$ for which $ X\neq  \underset{\sim}{R}(a_i,d_j) \supseteq \underset{\sim}{S}(a_i) \cup \underset{\sim}{F}(d_j) $ holds. Hence, $L((S,A)\cup (F,D))\supseteq L(S,A)\cup L(F,D))$.\\ 

\noindent
The proofs of (ii)-(iv) can be obtained easily. Hence, they are omitted. 

\end{proof}

\noindent
The following two theorems investigate the relationships among the complement of a soft set, the lower soft approximation, and the upper soft approximation. These results lay the foundational basis for studying the notion of roughness of a soft set via lower and upper soft approximations.

\begin{theorem}

Let $(X, \mathcal{R})$ be an approximation space,  $(S, A)$  and $(F,D)$ be two soft sets defined over $X$. Then, the following results hold:\\

(i) $L(S,A)= C(U(C(S, A)))$,\\

(ii) $U(S,A)= C(L(C(S, A)))$.\\

\end{theorem}

\begin{proof}
    (i) Let $C(S,A) = (W, A)$, where $W(a_i)= X\setminus S(a_i)$ $\forall a_i\in A$. Then, $U(C(S,A) = \{(a_i, \cup_k C_k) : a_i\in A, C_k \in X/\mathcal{R}, k\in \Delta, \Delta \ \text{ is an index set and} \ C_k \cap  W(a_i) \neq \emptyset \} = \{(a_i, \cup_k C_k) : a_i\in A, C_k \in X/\mathcal{R}, k\in \Delta, \Delta \ \text{ is an index set and} \ C_k \cap  (X\setminus S(a_i))\neq \emptyset \} =  C(L(S,A))$. Thus, $C(C(L(S,A)))= C(U(C(S,A)))$. Hence, $L(S,A)= C(U(C(S,A)))$.\\

\noindent
    The proof of (ii) can be easily obtained.
\end{proof}

\begin{theorem}

Let $(X, \mathcal{R})$ be an approximation space,  $(S, A)$  and $(F,D)$ be two soft sets. Then, the following results hold:\\

(i) $C(L(S,A)\cup L(F,D))= U(C(S,A))\cap U(C(F,D))$,\\

(ii) $C(U(S,A)\cup U(F,D))= L(C(S,A))\cap L(C(F,D))$,\\

(iii) $C(L(S,A)\cup U(F,D))= U(C(S,A))\cap L(C(F,D))$,\\

(iv)  $C(U(S,A)\cup L(F,D))= L(C(S,A))\cap U(C(F,D))$,\\

(V) $C(L(S,A)\cap L(F,D))= U(C(S,A))\cup U(C(F,D))$,\\

(vi) $C(U(S,A)\cap U(F,D))= L(C(S,A))\cup L(C(F,D))$,\\

(vii) $C(L(S,A)\cap U(F,D))= U(C(S,A))\cup L(C(F,D))$,\\

(viii)  $C(U(S,A)\cap L(F,D))= L(C(S,A))\cup U(C(F,D))$.\\

\end{theorem}

\begin{proof}
    We only prove (iii). The others can be obtained following similar techniques. \\

(iii) Using Theorem 3.7 and from 14, we get $C(L(S,A)\cup U(F,D))= C(L(S,A)) \cap C(U(F,D) = U(C(S,A))\cap L(C(F,D))$.   
\end{proof}

\section{Accuracy and roughness measures of soft sets }
The concept of an accuracy measure in rough set theory was first introduced by Pawlak \cite{6}. Its importance lies in quantifying the accuracy of a classical set defined on an approximation space in terms of its lower and upper approximations. Based on this measure, the roughness of a classical set can be determined, which is simply obtained by subtracting the accuracy value from $1$. In terms of applications, roughness measures play major roles in pattern recognition, image processing, etc. In 2005, Pal et al. \cite{46} used the roughness  measure of a subset of a universe to introduce rough entropy for images. 
Later, Sen and Pal \cite{47, 48} used roughness measures to define two entropy measures to study images. Later, Chakraborty et al. \cite{8} utilized the roughness measures of both an object and its background in an image separately to define entropy measures, thereby improving text image analysis. 
Recently, Pal \cite{11} highlighted the challenges in Granular Computing (GrC) and big data analytics, and discussed possible directions using rough set models, including roughness measures. Thus, the aforementioned facts motivate us to introduce the following two roughness measures. One of these measures is inspired by Pawlak \cite{6}, while the other is motivated by the axiomatic considerations of Yao \cite{30}. If $(S,A)\in (\emptyset, -)$, then $\tau(S,A)=\{\emptyset\}$. Hence, determining the accuracy measures for such soft sets becomes meaningless. Therefore, throughout the rest of this paper, such soft sets will not be considered unless explicitly stated otherwise.\\

\begin{definition}
    Let $(X, \mathcal{R})$ be an approximation space and $(S, A)$ be a soft set defined over $X$. Then, accuracy measure of $(S, A)$ with respect to $\mathcal{R}$ in the sense of Pawlak  is denoted by $\rho_{\mathcal{R}}^P(S,A)$ and is defined as $\rho_{\mathcal{R}}^P(S,A) =  \left( \frac{\sum_{a_i\in A}|\underset{\sim}{S}(a_i)| }{\sum_{a_i\in A}|\overset{\sim}{S}(a_i) |}\right)$, where $\tau(U(S,A))\neq \{\emptyset\}$.\\

\noindent
    Thus, roughness measure of $(S, A)$ with respect to $\mathcal{R}$ in the sense of Pawlak  is denoted by $\theta_{\mathcal{R}}^P(S,A)$ and is defined as $\theta_{\mathcal{R}}^P(S,A) =  1-\rho_{\mathcal{R}}^P(S,A)$.\\
\end{definition}

\begin{example}
   Consider Example 3.1. For the soft set $(S,A)$, we have $\rho_{\mathcal{R}}^P(S,A) = 0.4$, and thus $\theta_{\mathcal{R}}^P(S,A) = 0.6$.

\end{example}

\noindent
We now measure the range of roughness of a soft set that is defined over  an approximation space. \\

\begin{theorem}
    Let $(X, \mathcal{R})$ be an approximation space, and $(S, A)$ be a soft set defined over $X$.  Then, $0\leq \rho_{\mathcal{R}}^P(S,A) \leq 1$.
\end{theorem}

\begin{proof}
By Definition 3.1, we have 
$\underset{\sim}{S}(a_i) = \cup_j B_j$, where $B_j\subseteq S(a_i)$. So, it also implies $B_j\cap S(a_i) \neq \emptyset$.

\noindent
Again, by Definition 3.2, we have  
$\overset{\sim}{S}(a_i) = \cup_k C_k$, where $C_k\cap S(a_i) \neq \emptyset$. Thus, obviously,  $\emptyset \subseteq \underset{\sim}{S}(a_i) \subseteq \overset{\sim}{S}(a_i)$. Then, $0\leq \sum_{a_i\in A}|\underset{\sim}{S}(a_i)| \leq \sum_{a_i\in A}|\overset{\sim}{S}(a_i) |$. Hence,  $0\leq \rho_{\mathcal{R}}^P(S,A) \leq 1$.
\end{proof}

\noindent
From the aforementioned theorem, it follows that if the accuracy measure of a soft set $(S,A)$ defined over $X$ of an approximation space $(X,\mathcal{R})$ is $1$, then the cardinalities of lower and upper soft approximations of $\epsilon$-approximation sets of $(S,A)$ coincide. Consequently, for each $\epsilon$-approximation subset with $\epsilon \in A$, the corresponding lower and upper approximations are identical. Although a similar notion exists in rough set theory, the classical rough set framework does not permit an explicit analysis of overlapping phenomena. On the contrary, when the accuracy measure is zero, the lower approximation of each $\epsilon$-approximation subset is the empty set.  The following remark investigates the range of the roughness measure of a soft set.  \\

\begin{remark}
   Let $(X,\mathcal{R})$ be an approximation space, and $(S, A) $ be a
   soft set defined over $X$. Then,  
$0\leq \theta_{\mathcal{R}}^P(S,A) \leq 1$.
\end{remark}

\noindent
So far, we have examined the roughness measure of a soft set solely with respect to a single approximation space. It is therefore natural to investigate the impact of lower and upper soft approximations on a soft set when it is defined over the same universal set but under two different equivalence relations. Thus, we consider $\mathcal{R}_1$ and $\mathcal{R}_2$ be two equivalence relations on $X$. We denote $\mathcal{R}_1  \preceq \mathcal{R}_2$ to mean that $[x]_{\mathcal{R}_1} \subseteq [x]_{\mathcal{R}_2}$, for all $x\in X$.
This notion indicates that the granulation of $X$ with respect to $\mathcal{R}_1$ is finer than its granulation with respect to $\mathcal{R}_2$, respectively. Here, $[x]_{\mathcal{R}_1}$ and $[x]_{\mathcal{R}_2}$ indicate the equivalence classes of $x\in X$ with respect to the equivalence relations $\mathcal{R}_1$ and $\mathcal{R}_2$. To differentiate at the time of denoting,  we use $\underset{\sim}{{S}_{\mathcal{R}}}(a_i)$ and $\overset{\sim}{{S}_{\mathcal{R}}}(a_i)$ for $\underset{\sim}{S}(a_i)$ and $\overset{\sim}{S}(a_i)$, respectively, when the approximation space is considered as  $(X,R)$ and the soft set $(S,A)$ is defined over $X$. Thus, we obtain the following theorem:\\

\begin{theorem}
    Let $(X,\mathcal{R}_1)$ and $(X,\mathcal{R}_2)$ be two approximation spaces,  and $(S,A)$ be  a soft set defined over $X$.  If $\mathcal{R}_1\preceq \mathcal{R}_2$, then 
$\rho_{\mathcal{R}_2}^P(S,A)\leq \rho_{\mathcal{R}_1}^P(S,A)$.
 \end{theorem}

\begin{proof}
Since $\mathcal{R}_1\preceq \mathcal{R}_2 $, Thus, $[x]_{\mathcal{R}_1}\subseteq [x]_{\mathcal{R}_2} \forall x\in X$. Thus, obviously $\underset{\sim}{{S}_{\mathcal{R}_2}}(a_i) \subseteq \underset{\sim}{{S}_{\mathcal{R}_1}}(a_i) $ and $\overset{\sim}{{S}_{\mathcal{R}_1}}(a_i) \subseteq \overset{\sim}{{S}_{\mathcal{R}_2}}(a_i)$ $\forall a_i\in A$. Hence, $|\underset{\sim}{{S}_{\mathcal{R}_2}}(a_i)| \leq |\underset{\sim}{{S}_{\mathcal{R}_1}}(a_i) |$ and $|\overset{\sim}{{S}_{\mathcal{R}_1}}(a_i) |\leq |\overset{\sim}{{S}_{\mathcal{R}_2}}(a_i)|$. This yields $|\underset{\sim}{{S}_{\mathcal{R}_2}}(a_i)| \leq |\underset{\sim}{{S}_{\mathcal{R}_1}}(a_i) |$ and $\frac{1}{|\overset{\sim}{{S}_{\mathcal{R}_1}}(a_i) |}\geq \frac{1}{|\overset{\sim}{{S}_{\mathcal{R}_2}}(a_i)|}$.  
Hence, $\rho_{\mathcal{R}_2}^P(S,A)\leq \rho_{\mathcal{R}_1}^P(S,A)$.

\end{proof}

\noindent
The aforementioned theorem indicates that the accuracy measure of a soft set $(S,A)$ increases as the granulation of the universal set becomes finer. However, the following theorem shows that an increase in granularity, due to the formation of equivalence classes, results in a decrease in the roughness of the soft set. \\
 
\begin{corollary}
    Let $(X,\mathcal{R}_1)$ and $(X,\mathcal{R}_2)$ be two approximation spaces,  and $(S,A)$ be  a soft set defined on $X$.  If $\mathcal{R}_1\preceq \mathcal{R}_2$, then 
$\theta_{\mathcal{R}_1}^P(S,A)\leq \theta_{\mathcal{R}_2}^P(S,A)$.
\end{corollary}

\begin{proof}
    Using Definition 4.1 and Theorem 4.2, we have  $\theta_{\mathcal{R}_1}^P(S,A) = 1-\rho_{\mathcal{R}_1}^P(S,A) \leq 1-\rho_{\mathcal{R}_2}^P(S,A) =\theta_{\mathcal{R}_2}^P(S,A)$.
\end{proof}

\begin{theorem}
    Let $(X,\mathcal{R})$ be an approximation space,  and $(S,A)$ be  a soft set defined over $X$.  If $\rho_{\mathcal{R}}^P(S,A) = 1$, then $ L(S,A) = U(S,A)$.

\end{theorem}

\begin{proof}
If $\rho_{\mathcal{R}}^P(S,A) = 1$, then $\sum_{a_i\in A}|\underset{\sim}{S}(a_i)| = \sum_{a_i\in A}|\overset{\sim}{S}(a_i) |$. Since  $\underset{\sim}{S}(a_i) \subseteq \overset{\sim}{S}(a_i)$ $\forall a_i\in A$, thus $\sum_{a_i\in A}|\underset{\sim}{S}(a_i)| = \sum_{a_i\in A}|\overset{\sim}{S}(a_i) |$ is only possible when $\underset{\sim}{S}(a_i) =\overset{\sim}{S}(a_i)$ $\forall a_i\in A$. Hence, $L(S,A) =  U(S,A)$.
 
\end{proof}

\begin{corollary}
    Let $(X,\mathcal{R})$ be an approximation space,  and $(S,A)$ be  a soft set defined over $X$.  If $\theta_{\mathcal{R}}^P(S,A) = 0$, then $L(S,A) = U(S,A)$.

\end{corollary}

\begin{corollary}
    Let $(X,\mathcal{R})$ be an approximation space, and $(S,A)$ be  a soft set defined over $X$.  If $\rho_{\mathcal{R}}^P(S,A) = 0$, then $\tau( L(S,A)) = \{\emptyset\}$, and $\tau( U(S,A)) \neq \{\emptyset\}$ .
\end{corollary}
\begin{proof}
By Definition 4.1, we have $\tau( U(S,A)) \neq \{\emptyset\}$. Since, $\rho_{\mathcal{R}}^P(S,A) = 0$, thus $\sum_{a_i\in A}|\underset{\sim}{S}(a_i)| =0$. It yields $|\underset{\sim}{S}(a_i)|=0$ $\forall a_i\in A$. Thus, $\underset{\sim}{S}(a_i)= \emptyset $ $\forall a_i\in A$. Hence, $\tau( L(S,A)) = \{\emptyset\}$. 
\end{proof}

\noindent
Although Pawlak \cite{6} proposed an accuracy measure for a rough set, he did not specify any standard axiomatic criteria that could serve as a proper measure for assessing the accuracy of a rough set. Since Pawlak did not explicitly provide a definition or interpretation of the notions of `accuracy', `completeness', `quality', and `roughness', Yao \cite{30} subsequently proposed  axioms to characterize an accuracy measure of a rough set. His axioms are formulated in terms of rough sets associated with a subset of a universe. Motivated by this approach, we propose the following axioms for an accuracy measure of a soft set:\\

\begin{definition} (Axioms for an accuracy measure  of a soft set) Let $(X, \mathcal{R})$ be an approximation space, and  $(S,A)$ be a soft set defined over $X$. Then, the accuracy measure of $(S,A)$, say $\rho(S,A)$, is said to satisfy axioms for an  accuracy measure of a soft set if the following conditions hold: \\

(i) $\rho(S,A)=1 \iff L(S,A)= U(S,A)$, \\

(ii) $\rho(S,A)=0 \iff \tau (L(S,A))= \{\emptyset\}, \tau(U(S,A)) = \{X\}$, \\

(iii) For fixed value of $\sum_{a_i\in A}|\overset{\sim}{S}(a_i) |$; $\rho(S,A)$ is strictly increasing with $\sum_{a_i\in A}|\underset{\sim}{S}(a_i) |$,\\

(iv) For fixed value of $\sum_{a_i\in A}|\underset{\sim}{S}(a_i) |\neq 0$; $\rho(S,A)$ is strictly decreasing with $\sum_{a_i\in A}|\overset{\sim}{S}(a_i) |$,\\

(v) If $\mathcal{R}_1$ and $ \mathcal{R}_2$ are two equivalence relations on $X$ such that $\mathcal{R}_1\preceq \mathcal{R}_2$, then 
$\rho^{}_{{\mathcal{R}_2}}(S,A)\leq \rho^{}_{{\mathcal{R}_1}}(S,A)$.

\end{definition}

\noindent
It  is evident that $\rho_{\mathcal{R}}^P(S,A)$ does not satisfy the aforementioned axioms as a measurement of the accuracy of a soft set. Therefore, it is necessary to define an alternative accuracy measure of  a soft set that preserves these axioms. Thus, we define the following accuracy measure:\\

\begin{definition}
Let $(X, \mathcal{R})$ be an approximation space, and $(S, A)$ be a soft set defined over $X$. Then, accuracy measure of $(S, A)$ with respect to $\mathcal{R}$  is denoted by $\rho_{\mathcal{R}}^Y(S,A)$ and is defined  below:\\
\[
\rho_{\mathcal{R}}^Y(S,A)=
\begin{cases}
\dfrac{1}{2}\left(
\dfrac{\sum_{a_i\in A}\left|\underset{\sim}{S}(a_i)\right|}
{\sum_{a_i\in A}\left|S(a_i)\right|}
+
\dfrac{|X||A|-\sum_{a_i\in A}\left|\overset{\sim}{S}(a_i)\right|}
{|X||A|-\sum_{a_i\in A}\left|S(a_i)\right|}
\right),
& \text{if } \tau(S,A)\neq\{X\},\\[2ex]
\dfrac{1}{2}\left(
\dfrac{\sum_{a_i\in A}\left|\underset{\sim}{S}(a_i)\right|}
{\sum_{a_i\in A}\left|S(a_i)\right|}
+1
\right),
& \text{if } \tau(S,A)=\{X\}.
\end{cases}
\]\\

\noindent
In this case, roughness measure of $(S, A)$ with respect to $\mathcal{R}$  is denoted by $\theta_{\mathcal{R}}^Y(S,A)$ and is defined as $\theta_{\mathcal{R}}^Y(S,A) =  1-\rho_{\mathcal{R}}^Y(S,A)$. It can be seen that when $\tau(S,A)=\{X\}$, the values of $\rho_{\mathcal{R}}^P(S,A)$ and $\rho_{\mathcal{R}}^Y(S,A)$ coincide, and both are equal to $1$. Thus, we should focus on determining the accuracy measures of those soft sets $(S,A)$ for which $\tau(S,A)\neq \{X\}$. Thus, the following algorithm is provided to calculate $\rho_{\mathcal{R}}^Y(S,A)$, when $\tau(S,A)\neq \{X\}$:\\

\end{definition}

\begin{algorithm}[H]
\caption{Algorithm of the Yao's accuracy  measure of a soft set $(S,A)$ defined over $(X, \mathcal{R})$}
\label{alg:Y-roughness}
\begin{algorithmic}[1]
\STATE \textbf{Input:} $X$, equivalence relation $\mathcal R$, soft set $(S,A)$
\STATE \textbf{Output:} $\rho_{\mathcal R}^Y(S,A)$
\STATE Compute $X/\mathcal R$; initialize $L,U,M \gets 0$

\FOR{each $a_i \in A$}
    \STATE $\underset{\sim}{S}(a_i),\overset{\sim}{S}(a_i) \gets \emptyset$
    \FOR{each $C \in X/\mathcal R$}
        \IF{$C \subseteq S(a_i)$}
            \STATE $\underset{\sim}{S}(a_i) \gets \underset{\sim}{S}(a_i)\cup C$
        \ENDIF
        \IF{$C \cap S(a_i)\neq\emptyset$}
            \STATE $\overset{\sim}{S}(a_i) \gets \overset{\sim}{S}(a_i)\cup C$
        \ENDIF
    \ENDFOR
    \STATE $L \gets L+|\underset{\sim}{S}(a_i)|$,
           $U \gets U+|\overset{\sim}{S}(a_i)|$,
           $M \gets M+|S(a_i)|$
\ENDFOR

\STATE $\rho_{\mathcal R}^Y(S,A)
=\frac12\!\left(\frac{L}{M}+\frac{|X||A|-U}{|X||A|-M}\right)$
\STATE \textbf{Return} $\rho_{\mathcal R}^Y(S,A)$
\end{algorithmic}
\end{algorithm}

\noindent
Now, we consider the following example to determine the accuracy measures of a soft set $(S,A)$ and its complement $C(S,A)$ defined over $X$, and  $(X, \mathcal{R})$ is an approximation space. In Table 1, we list fifteen equivalence relations along with the corresponding values of $\rho_{\mathcal{R}}^P(S,A)$, $\rho_{\mathcal{R}}^P(C(S,A))$, $\rho_{\mathcal{R}}^Y(S,A)$, and $\rho_{\mathcal{R}}^Y(C(S,A))$.\\

\begin{example}
Let $(S,A )$ be a soft set defined on a universal set $X$, where 
$X=\{a,b,c,d\}, A=\{a_1, a_2, a_3\}$, $S(a_1)= \{a, b, d\}$, $S(a_2)= \{b, c, d\}$, and $S(a_3)= \{ c\}$. Thus, we obtain the following table: \\
\end{example}
\begin{table}[H]
\centering
\resizebox{1.0\textwidth}{!}{ 
\footnotesize
\setlength{\tabcolsep}{8pt} 
\renewcommand{\arraystretch}{1.2} 
\begin{tabular}{|c|l|c|c|c|c|}
\hline
\textbf{\makecell{Equivalence\\Relation}} 
& \textbf{\makecell{Equivalence Classes\\of $X$}} 
& $\boldsymbol{\rho_{\mathcal R_i}^{P}(S,A)}$ 
& $\boldsymbol{\rho_{\mathcal R_i}^{Y}(S,A)}$ 
& $\boldsymbol{\rho_{\mathcal R_i}^{P}(C(S,A))}$ 
& $\boldsymbol{\rho_{\mathcal R_i}^{Y}(C(S,A))}$ \\
\hline
$\mathcal{R}_1$ & $\{a,b,c,d\}$ & 0 & 0 & 0 & 0 \\
$\mathcal{R}_2$ & $\{a\},\{b,c,d\}$ & 0.4 & 0.48571 & 0.25 & 0.48571 \\
$\mathcal{R}_3$ & $\{b\},\{a,c,d\}$ & 0.18181 & 0.24285 & 0.1 & 0.24285 \\
$\mathcal{R}_4$ & $\{c\},\{a,b,d\}$ & 0.625 & 0.75714 & 0.57142 & 0.75714 \\
$\mathcal{R}_5$ & $\{d\},\{a,b,c\}$ & 0.18181 & 0.24285 & 0.1 & 0.24285 \\
$\mathcal{R}_6$ & $\{a,b\},\{c,d\}$ & 0.4 & 0.48571 & 0.25 & 0.48571 \\
$\mathcal{R}_7$ & $\{a,c\},\{b,d\}$ & 0.4 & 0.48571 & 0.25 & 0.48571 \\
$\mathcal{R}_8$ & $\{a,d\},\{b,c\}$ & 0.4 & 0.48571 & 0.25 & 0.48571 \\
$\mathcal{R}_9$ & $\{a,b\},\{c\},\{d\}$ & 0.66666 & 0.72857 & 0.66666 & 0.82857 \\
$\mathcal{R}_{10}$ & $\{a,c\},\{b\},\{d\}$ & 0.4 & 0.48571 & 0.25 & 0.48571 \\
$\mathcal{R}_{11}$ & $\{a,d\},\{b\},\{c\}$ & 0.75 & 0.82857 & 0.66666 & 0.82857 \\
$\mathcal{R}_{12}$ & $\{b,c\},\{a\},\{d\}$ & 0.55555 & 0.65714 & 0.42857 & 0.65714 \\
$\mathcal{R}_{13}$ & $\{b,d\},\{a\},\{c\}$ & 1 & 1 & 1 & 1 \\
$\mathcal{R}_{14}$ & $\{c,d\},\{a\},\{b\}$ & 0.555555 & 0.65714 & 0.42857 & 0.65714 \\
$\mathcal{R}_{15}$ & $\{a\},\{b\},\{c\},\{d\}$ & 1 & 1 & 1 & 1 \\
\hline
\end{tabular}
}
\caption{Calculated values of $\rho_{\mathcal{R}}^P(S,A)$, $\rho_{\mathcal{R}}^P(C(S,A))$, $\rho_{\mathcal{R}}^Y(S,A)$, and $\rho_{\mathcal{R}}^Y(C(S,A))$ for all equivalence relations $\mathcal R_i$, where $i\in \{1,2,3,..., 15\}$, considered on the universe $X=\{a,b,c,d\}$.}
\end{table}

\noindent
Since axioms for an accuracy measure of a soft set  are outlined in Definition 4.2, the following theorems are presented to examine whether the accuracy measure $\rho_{\mathcal{R}}^Y(S,A)$ satisfies these axioms.\\

\begin{proposition}
Let $(X, \mathcal{R})$ be an approximation space, and $(S, A)$ be a soft set defined over $X$ such that $\tau(S,A)\neq \{X\}$. Then, 
$\rho_{\mathcal{R}}^Y(S,A) =  1$ if and only if $L(S,A)= U(S,A)$.
\end{proposition}

\begin{proof}
 Given $\tau(S,A)\neq \{X\}$. If   
$\rho_{\mathcal{R}}^Y(S,A) =  1$, then we easily get $\sum_{a_i\in A}\left|\underset{\sim}{S}(a_i)\right|=
\sum_{a_i\in A}\left|S(a_i)\right|$ and $|X||A|-\sum_{a_i\in A}\left|\overset{\sim}{S}(a_i)\right|= |X||A|-\sum_{a_i\in A}\left|S(a_i)\right|$.
Now, $\sum_{a_i\in A}\left|\underset{\sim}{S}(a_i)\right|=
\sum_{a_i\in A}\left|S(a_i)\right|$ is only possible if $\left|\underset{\sim}{S}(a_i)\right|=
\left|S(a_i)\right| $$\forall a_i\in A$. It implies 
$\underset{\sim}{S}(a_i)=
S(a_i)  $ $\forall a_i\in A$. Hence, $L(S,A)= (S,A)$. Similarly, we get $U(S,A)= (S,A)$. Hence, $L(S,A)= U(S,A)$.\\

\noindent
On the other hand, let $L(S,A)=U(S,A)$.
Then,   we get  $L(S,A)=U(S,A)=(S,A)$. Then, 
$\underset{\sim}{S}(a_i)=\overset{\sim}{S}(a_i)=
S(a_i) $ $ \forall a_i\in A$. Then, $\rho_{\mathcal{R}}^Y(S,A) =  1$.
\end{proof}
\noindent

\begin{proposition}
Let $(X, \mathcal{R})$ be an approximation space, and $(S, A)$ be a soft set defined over $X$ such that $\tau(S,A)\neq \{X\}$. Then, 
$\rho_{\mathcal{R}}^Y(S,A) =  0$ if and only if $\tau(L(S,A))=\{\emptyset\}$ and $\tau(U(S,A))= \{X\}$.
\end{proposition}

\begin{proof}
Given  $\tau((S,A))\neq \{X\}$. Now, $\rho_{\mathcal{R}}^Y(S,A) =  0$ yields  
$\sum_{a_i\in A}\left|\underset{\sim}{S}(a_i)\right|=0$ and 
$|X||A|-\sum_{a_i\in A}\left|\overset{\sim}{S}(a_i)\right| =0$. It gives $\left|\underset{\sim}{S}(a_i)\right|=0$ and 
$|X||A|=\sum_{a_i\in A}\left|\overset{\sim}{S}(a_i)\right| $. So, we get $\underset{\sim}{S}(a_i)=\emptyset$ and $\overset{\sim}{S}(a_i)=X $ $\forall a_i\in A$. Thus, $\tau'(L(S,A))=\{\emptyset\}$ and $\tau'(U(S,A))= \{X\}$.\\

\noindent
Conversely, suppose that $\tau'(L(S,A))=\{\emptyset\}$ and $\tau'(U(S,A))= \{X\}$. Then, $\underset{\sim}{S}(a_i)=\emptyset$ and $\overset{\sim}{S}(a_i)= X  $ $\forall a_i\in A $. This is only possible when $S(a_i)\in P(X)\setminus \{\emptyset, X\} \forall a_i\in A$.
So, it is easy to obtain $\rho_{\mathcal{R}}^Y(S,A) =  0$.
\end{proof}

\begin{proposition}
Let $(X, \mathcal{R})$ be an approximation space, and $(S, A)$ be a soft set defined over $X$ such that $\tau(S,A)\neq \{X\}$. For fixed value of $\sum_{a_i\in A}|\overset{\sim}{S}(a_i) |$; $\rho_{\mathcal{R}}^Y(S,A)$ is strictly increasing with $\sum_{a_i\in A}|\underset{\sim}{S}(a_i) |$.
\end{proposition}

\begin{proof}
Given $\tau(S,A)\neq \{X\}$. Also, we are considering 
$\tau(S,A)\neq \{\emptyset\}$. Thus, $S(a_i)\neq \emptyset $ $\forall a_i\in A$ and $S(a_i)\neq X $ $\forall a_i\in A$. Hence, $S(a_i)\in P(X)\setminus \{\emptyset, X\}$ for some  $a_i\in A$. So, $\sum_{a_i\in A}\left|S(a_i)\right|$ and $|X||A|-\sum_{a_i\in A}\left|S(a_i)\right|$ are fixed. Moreover, $\sum_{a_i\in A}|\overset{\sim}{S}(a_i) |$ is also fixed.  Thus, $\rho_{\mathcal{R}}^Y(S,A)$ is strictly increasing with $\sum_{a_i\in A}|\underset{\sim}{S}(a_i) |$.
\end{proof}

\begin{proposition}
Let $(X, \mathcal{R})$ be an approximation space, and $(S, A)$ be a soft set defined over $X$ such that $\tau(S,A)\neq \{X\}$. For fixed $\sum_{a_i\in A}|\underset{\sim}{S}(a_i) | \neq 0$; $\rho_{\mathcal{R}}^Y(S,A)$ is strictly  decreasing with $\sum_{a_i\in A}|\overset{\sim}{S}(a_i) |$.
\end{proposition}
\begin{proof}
Given $\tau(S,A)\neq \{X\}$. Also, we are considering 
$\tau(S,A)\neq \{\emptyset\}$. Thus, $S(a_i)\neq \emptyset $ $\forall a_i\in A$ and $S(a_i)\neq X $ $\forall a_i\in A$. Hence, $S(a_i)\in P(X)\setminus \{\emptyset, X\}$ for some  $a_i\in A$.  Now, $\sum_{a_i\in A}|\underset{\sim}{S}(a_i) | \neq 0$ yields $|\underset{\sim}{S}(a_i) | \neq 0 $ $\forall a_i\in A$. Thus, $\underset{\sim}{S}(a_i)  \neq \emptyset$ for some $a_i\in A$. So, $\sum_{a_i\in A}\left|S(a_i)\right|$ and $|X||A|-\sum_{a_i\in A}\left|S(a_i)\right|$ are fixed. Thus, $\rho_{\mathcal{R}}^Y(S,A)$ is strictly decreasing with $\sum_{a_i\in A}|\overset{\sim}{S}(a_i) |$.
\end{proof}
\begin{proposition}
Let  $(S, A)$ be a soft set defined over a universe $X$ such that $\tau'(S,A)\neq \{\emptyset\}$ and $\tau'(S,A)\neq \{X\}$. If $\mathcal{R}_1$ and $ \mathcal{R}_2$ are two equivalence relations on $X$ such that $\mathcal{R}_1\preceq \mathcal{R}_2$, then 
$\rho^{Y}_{{\mathcal{R}_2}}(S,A)\leq \rho^{Y}_{{\mathcal{R}_1}}(S,A)$
\end{proposition}
  \begin{proof}
Since $\tau'(S,A)\neq \{\emptyset\}$ and $\tau'(S,A)\neq \{X\}$, thus $\tau'(C(S,A))\neq \{X\}$ and $\tau'(C(S,A))\neq \{\emptyset\}$. Thus, $S(a_i)\in P(X)\setminus \{\emptyset, X\} \forall a_i\in A$. If $\mathcal{R}_1\preceq \mathcal{R}_2$, then $|\underset{\sim}{{S}_{\mathcal{R}_2}}(a_i)| \leq |\underset{\sim}{{S}_{\mathcal{R}_1}}(a_i) |$ and $|\overset{\sim}{{S}_{\mathcal{R}_1}}(a_i) |\leq |\overset{\sim}{{S}_{\mathcal{R}_2}}(a_i)|$. Hence, it can be easily found that $\rho^{Y}_{{\mathcal{R}_2}}(S,A)\leq \rho^{Y}_{{\mathcal{R}_1}}(S,A)$.
\end{proof}  

\begin{theorem}
Let $(S, A)$ be a soft set defined over $X$ such that $\tau'(S,A)\neq \{\emptyset\}$ and $\tau'(S,A)\neq \{X\}$. Then, $\rho^{Y}_{{\mathcal{R}}}(S,A)$ satisfies axioms of an accuracy measure of a soft set. 
\end{theorem}
\begin{proof}
The proof can be obtained by considering propositions 4.1 to 4.5 collectively. 
\end{proof}

\noindent
Consider Example 3.1. If  $C(S,A)$ is the complement of the  soft set $(S,A)$, then it can be found that $\rho^{Y}_{{\mathcal{R}_6}}(C(S,A)) = 0.25$. On the other hand, from Table 1, it is found that $\rho^{Y}_{{\mathcal{R}_6}}(S,A) = 0.4$. Thus, we can conclude the following remark:\\

\begin{remark}

Let $(X,R) $ be an approximation space and $(S,A)$ be a soft set defined  over $X$. Then,  $\rho^{P}_{{\mathcal{R}}}(S,A) \neq \rho^{P}_{{\mathcal{R}}}(C(S,A)) $ in general, but  $\rho^{Y}_{{\mathcal{R}}}(S,A) = \rho^{Y}_{{\mathcal{R}}}(C(S,A)) $ is always true. Thus, it is obvious to us that $\theta^{P}_{{\mathcal{R}}}(S,A) \neq \theta^{P}_{{\mathcal{R}}}(C(S,A)) $ in general, but $\theta^{Y}_{{\mathcal{R}}}(S,A) = \theta^{Y}_{{\mathcal{R}}}(C(S,A)) $ is always true.
    
\end{remark}

\noindent
Figure 4(a) depicts the accuracy measures of $(S,A)$, as defined in Definitions 4.1 and 4.3 for fifteen equivalence relations presented in Example 4.2. The blue boxes represent the values of the accuracy measures of $(S,A)$  as defined in Definition 4.1, whereas the red boxes correspond to the accuracy measures of $(S,A)$  as defined in Definition 4.3. Similarly, Figure 4(b) illustrates the accuracy measures of $C(S,A)$, the complement of the soft set $(S,A)$. From both figures, it is evident that the accuracy measure of Definition 4.3 yields a higher accuracy compared to the accuracy measure of Definition 4.1, both for a soft set and its complement. Thus, it can be concluded that if one may wish to study the accuracy measure of a soft set without depending on its complement, Definition 4.1 can be used, but it should not be expected that it will be the best accuracy measure connecting a soft set and its complement. Moreover, in general, the accuracy measures of a soft set $(S,A)$ and its complement $C(S,A)$ need not be identical. Consequently, from an application-oriented perspective, some loss of information may inevitably occur if one calculates the accuracy measure of a soft set of Definition 4.1. However, such situations do not arise when the accuracy measure is defined in Definition 4.3.

\begin{figure}[h!]
\centering

\begin{subfigure}[b]{0.9\textwidth}
    \centering
    \begin{tikzpicture}
    \begin{axis}[
        width=\textwidth,
        height=6cm,
        xlabel={Equivalence Relations},
        ylabel={Accuracy measures},
        xtick=data,
        xticklabels={
            $\mathcal{R}_1$, $\mathcal{R}_2$, $\mathcal{R}_3$, $\mathcal{R}_4$, $\mathcal{R}_5$,
            $\mathcal{R}_6$, $\mathcal{R}_7$, $\mathcal{R}_8$, $\mathcal{R}_9$, $\mathcal{R}_{10}$,
            $\mathcal{R}_{11}$, $\mathcal{R}_{12}$, $\mathcal{R}_{13}$, $\mathcal{R}_{14}$, $\mathcal{R}_{15}$
        },
        ymin=0, ymax=1.0,
        tick label style={font=\small},
        legend style={at={(0.97,0.03)}, anchor=south east, font=\small, draw=none}
    ]
    \addplot[color=blue, mark=o, thick] coordinates {
        (1,0) (2,0.4) (3,0.18181) (4,0.625) (5,0.18181)
        (6,0.4) (7,0.4) (8,0.4) (9,0.66666) (10,0.4)
        (11,0.75) (12,0.55555) (13,1) (14,0.555555) (15,1)
    };
    \addlegendentry{$\rho_{\mathcal R_i}^{P}(S,A)$}

    \addplot[color=red, mark=square*, thick] coordinates {
        (1,0) (2,0.48571) (3,0.24285) (4,0.75714) (5,0.24285)
        (6,0.48571) (7,0.48571) (8,0.48571) (9,0.72857) (10,0.48571)
        (11,0.82857) (12,0.65714) (13,1) (14,0.65714) (15,1)
    };
    \addlegendentry{$\rho_{\mathcal R_i}^{Y}(S,A)$}
    \end{axis}
    \end{tikzpicture}
    \caption{}
\end{subfigure}

\vspace{1em} 

\begin{subfigure}[b]{0.9\textwidth}
    \centering
    \begin{tikzpicture}
    \begin{axis}[
        width=\textwidth,
        height=6cm,
        xlabel={Equivalence Relations},
        ylabel={Accuracy measures},
        xtick=data,
        xticklabels={
            $\mathcal{R}_1$, $\mathcal{R}_2$, $\mathcal{R}_3$, $\mathcal{R}_4$, $\mathcal{R}_5$,
            $\mathcal{R}_6$, $\mathcal{R}_7$, $\mathcal{R}_8$, $\mathcal{R}_9$, $\mathcal{R}_{10}$,
            $\mathcal{R}_{11}$, $\mathcal{R}_{12}$, $\mathcal{R}_{13}$, $\mathcal{R}_{14}$, $\mathcal{R}_{15}$
        },
        ymin=0, ymax=1.0,
        tick label style={font=\small},
        legend style={at={(0.97,0.03)}, anchor=south east, font=\small, draw=none}
    ]
    \addplot[color=blue, mark=o, thick] coordinates {
        (1,0) (2,0.25) (3,0.1) (4,0.57142) (5,0.1)
        (6,0.25) (7,0.25) (8,0.25) (9,0.66666) (10,0.25)
        (11,0.66666) (12,0.42857) (13,1) (14,0.42857) (15,1)
    };
    \addlegendentry{$\rho_{\mathcal R_i}^{P}(C(S,A))$}

    \addplot[color=red, mark=square*, thick] coordinates {
        (1,0) (2,0.48571) (3,0.24285) (4,0.75714) (5,0.24285)
        (6,0.48571) (7,0.48571) (8,0.48571) (9,0.82857) (10,0.48571)
        (11,0.82857) (12,0.65714) (13,1) (14,0.65714) (15,1)
    };
    \addlegendentry{$\rho_{\mathcal R_i}^{Y}(C(S,A))$}
    \end{axis}
    \end{tikzpicture}
    \caption{}
\end{subfigure}

\caption{(a) Graphs of accuracy measures of $\rho_{\mathcal R_i}^{P}(S,A)$ and $\rho_{\mathcal R_i}^{Y}(S,A)$ for all equivalence relations $\mathcal R_i$, where $i\in \{1,2,3,..., 15\}$, on $X=\{a,b,c,d\}$ ; (b) Graphs of accuracy measures of $\rho_{\mathcal R_i}^{P}(C(S,A))$ and $\rho_{\mathcal R_i}^{Y}(C(S,A))$ for all equivalence relations $\mathcal R_i$, where $i\in \{1,2,3,..., 15\}$ on $X=\{a,b,c,d\}$. }
\end{figure}
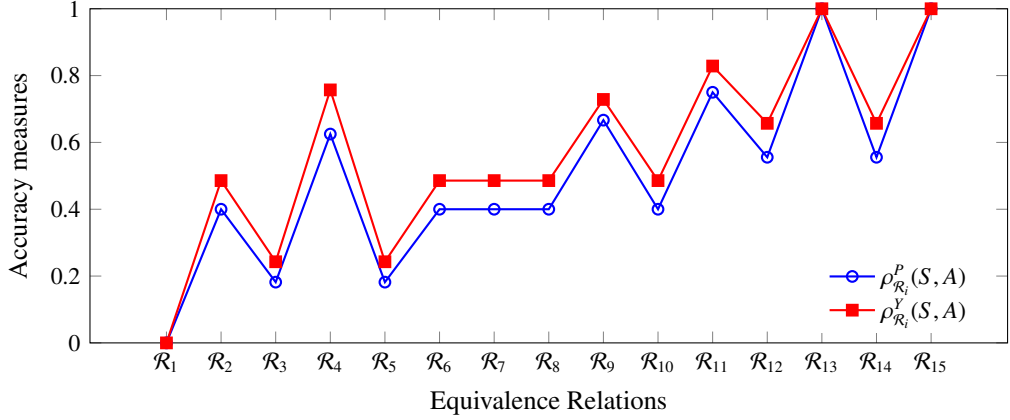
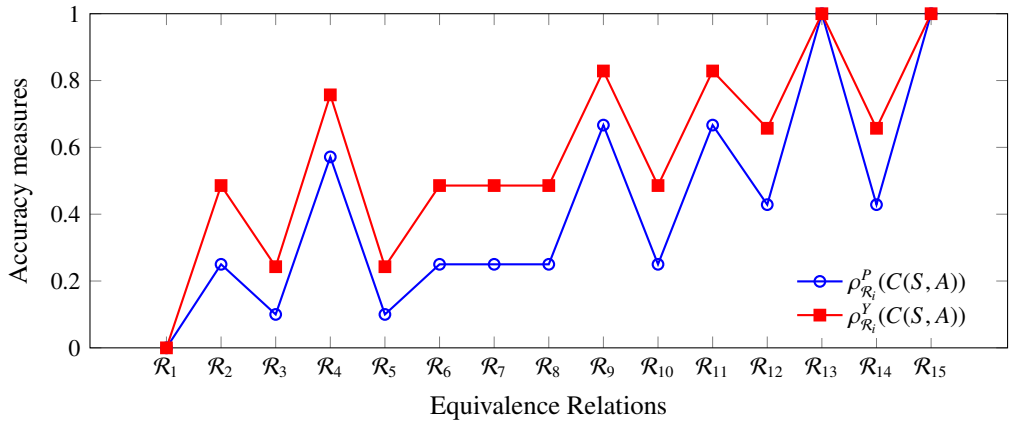
\section{Entropy of a soft set}
From the above section, it is observed that the accuracy measure of a soft set defined in Definition 4.3 performs better than the corresponding accuracy measure defined in Definition 4.1. If the accuracy measure of a soft set is greater than that of another soft set, then the roughness measure of the former is smaller than that of the latter. This is true for both Definitions 4.1 and 4.3. It is worth noting that the roughness measure of a rough set has played a significant role in the formulation of various entropy measures. In \cite{31}, Marek and Pawlak established the mathematical foundations of information storage and retrieval. Entropy is one of the most fundamental concepts in information theory, image processing, pattern recognition, and the thermodynamics of physics. In information theory, the idea of entropy was laid by Shannon \cite{32}. D{\"u}ntsch and Gediga \cite{33} introduced the entropy on the equivalence relation $\mathcal{R}$ of an approximation space $(X, \mathcal{R})$ for the first time using the concept of the probability distribution of elements over equivalence classes with respect to a rough set.  Later, Wierman \cite{34} introduced measure of information in rough set theory. Since then, many authors have introduced various measures of uncertainty \cite{35,36,37,38,39,40,41}. In these approaches, the lower and upper approximations play  relatively minor roles in rough set theory and hybrid structures. However, more recently, Tang et al. \cite{41} addressed this limitation and proposed an uncertain entropy measure on the partition of the universal set. However, these uncertainty measures are combinatorial in nature rather than logarithmic \cite{32,42} or exponential \cite{43,44,45}, which are widely used in entropy calculations within rough set theory for various real-life applications. Pal and Pal \cite{42} introduced a generalized logarithmic entropy measure for the analysis of image sets. Subsequently, Pal and Pal \cite{44} proposed a new entropy measure, called exponential entropy, in the context of image processing for object-background segmentation. They later investigated various properties of this entropy measure in \cite{43,44,45} in various problems related to  image processing, pattern recognition, etc. However, none of the works in \cite{42,43,44,45} studied these newly defined entropy measures within the framework of rough set theory. For the first time, the entropy measure in terms of logarithm was introduced by Pal et al. \cite{46} to address the problem of image-object extraction in the frame work of rough set and granular computing. 
Subsequently, new kinds of logarithmic entropy measures, along with exponential entropy measures, were introduced by Sen and Pal \cite{47,48} to quantify ambiguity in images using rough set theory. Both logarithmic entropy and exponential entropy are very useful in pattern recognition and image processing, particularly in relation to rough set theory. Rough entropy defined using logarithms is utilized to measure ambiguities in images and to extract object regions from an image by minimizing the roughness of both the object and its background \cite{46,47,48}. On the other hand, rough entropy using exponential behavior serves as an alternative mathematical tool to measure the aforementioned aspects. In this section, we introduce six types of entropy measures to quantify the incompleteness of knowledge associated with a soft set defined over a universe. The first three are in the modified sense of Shannon's entropy using logarithms, while the remaining three are in the modified sense of exponential entropy.\\

\begin{definition}
 Let $(X,\mathcal{R})$ be an approximation space and $(S,A)$ be a soft set defined over $X$. Then, the First Pawlak entropy of $(S,A)$  is denoted by  $Ent^{1P}_e(S,A)$ and defined as $Ent^{1P}_e(S,A)= - \frac{e}{2}[\theta_{\mathcal{R}}^P(S,A)\log_e(\theta_{\mathcal{R}}^P(S,A))+\theta_{\mathcal{R}}^P(C(S,A))\log_e(\theta_{\mathcal{R}}^P(C(S,A)))].$
\end{definition}
\noindent
In the aforementioned definition, the terms ``gain in incompleteness" are  $-\log_e(\theta_{\mathcal{R}}^P(S,A))$ and $-\log_e(\theta_{\mathcal{R}}^P(C(S,A)))$. They are used to measure $Ent^{1P}_e(S,A)$ using 
$\theta_{\mathcal{R}}^P(S,A)$ and $ \theta_{\mathcal{R}}^P(C(S,A)$. Similarly, we can define the following entropy measures analogous to Sen and Pal \cite{47}, and Pal and Pal \cite { 45, 46}.\\

\begin{definition}
 Let $(X,\mathcal{R})$ be an approximation space and $(S,A)$ be a soft set defined over $X$. Then, the Second Pawlak  entropy of $(S,A)$  is denoted by  $Ent^{2P}_{\beta}(S,A)$ and defined as $Ent^{2P}_{\beta}(S,A)= - \frac{1}{2}[\theta_{\mathcal{R}}^P(S,A)\log_{\beta}(\frac{\theta_{\mathcal{R}}^P(S,A)}{\beta})+\theta_{\mathcal{R}}^P(C(S,A))\log_{\beta}(\frac{\theta_{\mathcal{R}}^P(C(S,A))}{\beta})].$ 
\end{definition}

\begin{definition}
 Let $(X,\mathcal{R})$ be an approximation space and $(S,A)$ be a soft set defined over $X$. Then, the first Pawlak exponential entropy of $(S,A)$  is denoted by  $Ent^{P}_{\beta}(S,A)$ and defined as $Ent^{P}_{\beta}(S,A)= \frac{1}{2}[\theta_{\mathcal{R}}^P(S,A){\beta}^{(1-\theta_{\mathcal{R}}^P(S,A))}+
 \theta_{\mathcal{R}}^P(C(S,A)){\beta}^{(1-\theta_{\mathcal{R}}^P(C(S,A)))}].$
\end{definition}

\begin{definition}
 Let $(X,\mathcal{R})$ be an approximation space and $(S,A)$ be a soft set defined over $X$. Then, the Third Pawlak entropy of $(S,A)$  is denoted by  $Ent^{3P}_e(S,A)$ and defined as $Ent^{3P}_e(S,A)= - \theta_{\mathcal{R}}^Y(S,A)\log_e(\theta_{\mathcal{R}}^Y(S,A)).$
\end{definition}

\begin{definition}
 Let $(X,\mathcal{R})$ be an approximation space and $(S,A)$ be a soft set defined over $X$. Then, the Fourth Pawlak  entropy of $(S,A)$  is denoted by  $Ent^{4P}_{\beta}(S,A)$ and defined as $Ent^{4P}_{\beta}(S,A)= - \theta_{\mathcal{R}}^Y(S,A)\log_{\beta}(\frac{\theta_{\mathcal{R}}^Y(S,A)}{\beta}),$.
\end{definition}

\begin{definition}
 Let $(X,\mathcal{R})$ be an approximation space and $(S,A)$ be a soft set defined over $X$. Then, the second Pawlak exponential entropy of $(S,A)$  is denoted by  $Ent^{'P}_{\beta}(S,A)$ and defined as $Ent^{'P}_{\beta}(S,A)= \theta_{\mathcal{R}}^Y(S,A){\beta}^{(1-\theta_{\mathcal{R}}^Y(S,A))}$.
\end{definition}

\noindent
Throughout this paper, we adopt the convention $0 .\log_e(0)=0$ for computational purposes. Moreover, this convention has also been used in \cite{47,48} and many other works in pattern recognition, image processing, etc. This convention does not hamper our analysis from the computational perspective, as observed in various works available in the literature. We now discuss the properties of the aforementioned entropy measures. Theorem 5.1 examines the properties, namely maximum value, continuity, entropic symmetry, and concavity, of $Ent^{1P}_e(S,A)$. \\

\begin{theorem}
Let $(X,\mathcal{R})$ be an approximation space, and $(S,A)$ be a soft set defined over $X$. Then,  $Ent^{1P}_e(S,A)$ satisfies the following results:\\

(i) The maximum value of $Ent^{1P}_e(S,A) $ is $1$ for $ \theta_{\mathcal{R}}^P(S,A) = \theta_{\mathcal{R}}^P(C(S,A)) = 1/e$, and the minimum value is zero when $\theta_{\mathcal{R}}^P(S,A), \theta_{\mathcal{R}}^P(C(S,A))\in \{0,1\}$;\\

(ii) $Ent^{1P}_e(S,A) $ is a continuous function in [0,1];\\

(iii) $Ent^{1P}_e(S,A) = Ent^{1P}_e(C(S,A)) $;\\

(iv) $Ent^{1P}_e(S,A)$ is a concave function in [0,1]\\

\end{theorem}

\begin{proof}
    We know that $\theta_{\mathcal{R}}^P(S,A)\in [0,1]$ and $\theta_{\mathcal{R}}^P(C(S,A))\in [0,1]$. Moreover, we consider the convention $0\cdot \log_e(0)=0$. Thus, we are to check for $\theta_{\mathcal{R}}^P(S,A)\in (0,1]$ and $\theta_{\mathcal{R}}^P(C(S,A))\in (0,1]$. Let $x= \theta_{\mathcal{R}}^P(S,A) $, $y= \theta_{\mathcal{R}}^P(C(S,A)) $, and $z=Ent^{1P}_e(S,A) $. So, we get $x, y\in (0,1]$. \\
   
 (i)   Then,  $z= - \frac{e}{2}[x\log_e(x)+y\log_e(y)].$ Hence, $
\frac{\partial z}{\partial x}= - \frac{e}{2}[1+\log_e(x)],  \frac{\partial z}{\partial y}= - \frac{e}{2}[1+\log_e(y)], 
\frac{\partial^2 z}{\partial x^2}= - \frac{e}{2x}< 0, \frac{d^2 z}{dy^2}= - \frac{e}{2y}< 0$ and $\frac{\partial ^2 z}{\partial x \partial y}= 0$. Thus, $\frac{\partial z}{\partial x}= 0 \implies x=1/e$ and $\frac{\partial z}{\partial y}= 0 \implies y=1/e$.\\

Now, $\left( \frac{\partial^2 z}{\partial x \partial y}  \right)^2- (\frac{\partial ^2 z}{\partial x^2})(\frac{\partial^2 z}{\partial y^2}) = - \frac{e^2}{4xy}<0$. Hence, $z$ is a concave function. Moreover, maximum value of $z$ is  $1$ and this occurs when $ \theta_{\mathcal{R}}^P(S,A) = \theta_{\mathcal{R}}^P(C(S,A)) = 1/e$.\\

(ii) Let $f(x)= x\log_e(x)$ and $g(x)=y\log_e(y)$ be two functions where $x,y\in (0,1]$.  Then, $f(x)$ and $g(y)$ are both continuous functions in (0,1]. Moreover, we assumed the convention $0.\log_e(0)=0$. Then, $f(x)+g(y)$ is also a continuous function in [0,1]. Thus, $z$ is a continuous function in [0,1].\\

(iii) We know that $C(C(S,A))= (S,A)$. Thus, $Ent^{1P}_e(S,A) = Ent^{1P}_e(C(S,A)) $. \\

(iv) The concavity  of $Ent^{1P}_e(S,A)$ is shown in 
the proof of (i).
  
\end{proof}
\noindent
Figure 5 shows, using a 3D surface plot,  the entropy measure $Ent^{1P}_e(S,A)$ for the soft set $(S,A)$ as defined in Example 4.2 .  In this figure, the values of $\theta_{\mathcal{R}}^P(S,A)$ and $\theta_{\mathcal{R}}^P(C(S,A))$ are considered within the interval $[0,1]$. To make the computation feasible, we use Python  3.14. and consider $0.\log_e(0)=0$, otherwise it would not be possible to obtain the 3D graph of these entropy measures wherever necessary.  It can be verified that the maximum value of $Ent^{1P}_e(S,A)$ is $1$, and this maximum occurs at the point  $(1/e,\,1/e, 1)$ in this 3D surface plot.  The point at which $Ent^{1P}_e(S,A)$ attains its maximum value is marked by a red-colored point. It is important to note that the minimum value $0$ is obtained at the four corners $(0,0, 0)$, $(1,0, 0)$, $(0,1,0)$ and $(1,1,1)$ using the assumption $0.\log_e(0)=0$. In Theorem 5.2, we show that when the granulation of $X$ is finer with respect to an equivalence relation $\mathcal{R}_1$ in comparison to another equivalence relation $\mathcal{R}_2$, then the entropy $Ent^{1P}_e(S,A)$ is more with respect to $\mathcal{R}_1$ than that of $\mathcal{R}_2$. 

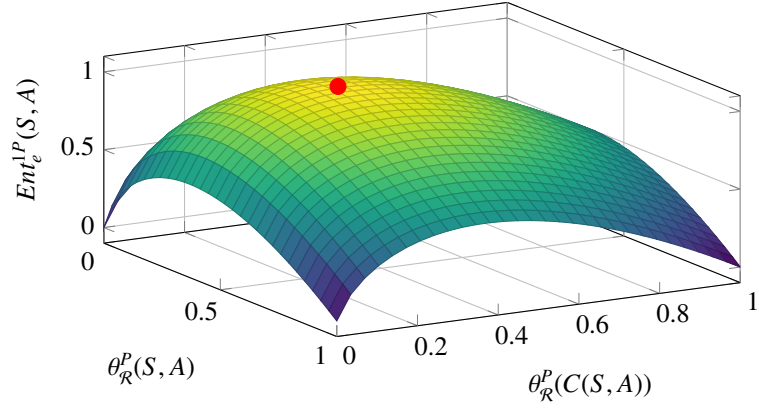
\begin{figure}[H]
\centering
\begin{tikzpicture}
\begin{axis}[
    width=10cm,
    height=6cm,
    xlabel={$\theta_{\mathcal{R}}^P(S,A)$},
    ylabel={$\theta_{\mathcal{R}}^P(C(S,A))$},
    zlabel={$Ent^{1P}_e(S,A)$},
    grid=major,
    view={60}{30},
]

\addplot3[
    surf,
    samples=30,
    domain=0:1,
    y domain=0:1,
    colormap/viridis,
]
(
    x,
    y,
    {
        -exp(1)/2*(
            (x==0 ? 0 : x*ln(x)) +
            (y==0 ? 0 : y*ln(y))
        )
    }
);

\addplot3[
    only marks,
    mark=*,
    mark size=3pt,
    color=red
]
coordinates {(0.3679,0.3679,1)};

\end{axis}
\end{tikzpicture}
\caption{3D surface plot of $Ent^{1P}_e(S,A)$  with its maximum value obtained at the point  $(1/e,1/e, 1)$.}
\end{figure}

\begin{theorem}
    Let $(X,\mathcal{R}_1)$ and $(X,\mathcal{R}_2)$ be two approximation spaces,  and $(S,A)$ be  a soft set defined over $X$.  If $\mathcal{R}_1\preceq \mathcal{R}_2$, then $(Ent^{1P}_e(S,A))_{\mathcal{R}_2} \leq (Ent^{1P}_e(S,A))_{\mathcal{R}_1} $.\\

\end{theorem}

\begin{proof}
    If $\mathcal{R}_1\preceq \mathcal{R}_2$, by Corollary 4.1,  $\theta_{\mathcal{R}_1}^P(S,A)\leq \theta_{\mathcal{R}_2}^P(S,A)$ and $\theta_{\mathcal{R}_1}^P(C(S,A))\leq \theta_{\mathcal{R}_2}^P(C(S,A))$. So, we get that $ \theta_{\mathcal{R}_1}^P(S,A)\log_e(\theta_{\mathcal{R}_1}^P(S,A))\leq \theta_{\mathcal{R}_2}^P(S,A)\log_e(\theta_{\mathcal{R}_2}^P(S,A))$ and  $\theta_{\mathcal{R}_1}^P(C(S,A))\log_e(\theta_{\mathcal{R}_1}^P(C(S,A))\leq \theta_{\mathcal{R}_2}^P(C(S,A))\log_e(\theta_{\mathcal{R}_2}^P(C(S,A))$. Then, $(Ent^{1P}_e(S,A))_{\mathcal{R}_2} \leq Ent^{1P}_e(S,A))_{\mathcal{R}_1} $. 
  \end{proof}

\noindent
We study the basic properties of  entropy measure $Ent^{2P}_{\beta}(S,A)$ in Theorem 5.3.  Here, the terms  ``gain in incompleteness" are  $-\log_{\beta}(\frac{\theta_{\mathcal{R}}^P(S,A)}{\beta})$ and $-\log_{\beta}(\frac{\theta_{\mathcal{R}}^P(C(S,A))}{\beta})$. Analogous to the definition of entropy  introduced in Sen and Pal \cite{48}, the values of $-\log_{\beta}(\frac{\theta_{\mathcal{R}}^P(S,A)}{\beta})$ and $-\log_{\beta}(\frac{\theta_{\mathcal{R}}^P(C(S,A)}{\beta})$ lie in $[1,\infty]$  for $\beta>1$. We can easily obtain the proof of the following theorem. So, we skip it.\\

\begin{theorem}
Let $(X,\mathcal{R})$ be an approximation space, and $(S,A)$ be a soft set defined over $X$. Then,  $Ent^{2P}_{\beta}(S,A)$ satisfies the following results:\\

(i) The maximum value of $Ent^{2P}_{\beta}(S,A)$ is given below:

\[
\max_{\substack{\theta_{\mathcal{R}}^P(S,A)\in[0,1] \\ \theta_{\mathcal{R}}^P(C(S,A))\in[0,1]}} 
Ent^{2P}_{\beta}(S,A)
=
\begin{cases}
\dfrac{\beta}{e\,\log_e \beta}, & \text{when } 1<\beta \le e \text{ and attained for } \theta_{\mathcal{R}}^P(S,A)=\theta_{\mathcal{R}}^P(C(S,A))=\dfrac{\beta}{e},\\[1.2ex]
1, & \text{when } \beta > e \text{ and attained for } \theta_{\mathcal{R}}^P(S,A)=\theta_{\mathcal{R}}^P(C(S,A))=1;
\end{cases}
\]

(ii) $Ent^{2P}_{\beta}(S,A)$ is a continuous function in [0,1];\\

(iii) $Ent^{2P}_{\beta}(S,A) = Ent^{2P}_{\beta}(C(S,A)) $;\\

(iv) $Ent^{2P}_{\beta}(S,A)$ is a concave function in [0,1]\\

\end{theorem}

\noindent
Figure 6 shows the entropy measure $Ent^{2P}_{\beta}(S,A)$  when the values of $\theta_{\mathcal{R}}^P(S,A)$ and $\theta_{\mathcal{R}}^P(C(S,A))$ are considered within the interval $[0,1]$. In this figure, the maximum value of $Ent^{2P}_{\beta}(S,A)$  is obtained as $1$ when $\theta_{\mathcal{R}}^P(S,A)=1$, $\theta_{\mathcal{R}}^P(C(S,A))=1$
and $\beta= e$. The maximum value is depicted using red-color point. This computational value also satisfies condition (i) of the aforementioned theorem. It is important to note that the maximum value of $Ent^{2P}_{\beta}(S,A)$ depends on the choice of $\beta$. When $\beta$ exceeds $e$, then we obtain the maximum value 1 since in both the cases of   values of $\theta_{\mathcal{R}}^P(S,A) $ and $\theta_{\mathcal{R}}^P(C(S,A))$, the maximum value is $1$. Theorem 5.4 examines the behavior of $Ent^{2P}_{\beta}(S,A)$ in the case where, for a fixed value of $\beta$, an equivalence relation ${\mathcal{R}_1}$ on $X$ is finer than another equivalence relation ${\mathcal{R}_2}$ of $X$. In contrast, Theorem 5.5 examines how $Ent^{2P}_{\beta}(S,A)$ behaves for two distinct values of $\beta$.\\

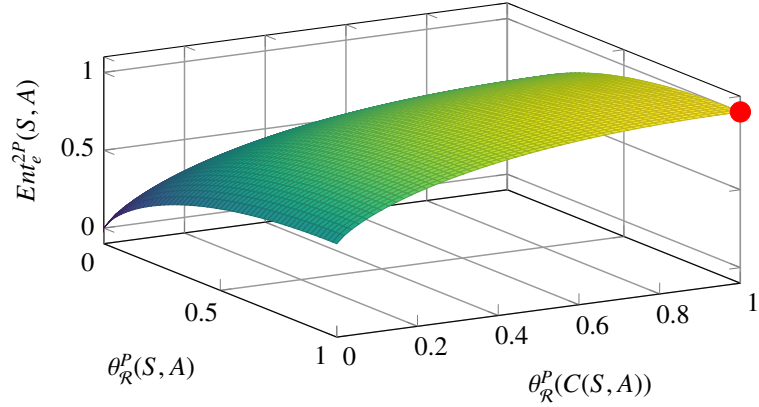
\begin{figure}[H]
\centering
\begin{tikzpicture}
\begin{axis}[
    width=10 cm, height=6cm,   
    view={60}{30},
    xlabel={$\theta_{\mathcal{R}}^P(S,A)$},
    ylabel={$\theta_{\mathcal{R}}^P(C(S,A))$},
    zlabel={$Ent^{2P}_e(S,A)$},
    colormap/viridis,
    domain=0:1,
    y domain=0:1,
    samples=55,
    samples y=55,
    z buffer=sort,
    grid=both,
    major grid style={line width=0.55pt,draw=gray!80},
    minor grid style={line width=0.28pt,draw=gray!40},
]

\addplot3[surf]
(
    {x},
    {y},
    { -0.5*((x==0?0:x*ln(x)) + (y==0?0:y*ln(y))) + 0.5*(x+y) }
);

\addplot3[
    only marks,
    mark=*,
    mark size=3.8pt,
    color=red
] coordinates {(1,1,1)};

\end{axis}
\end{tikzpicture}
\caption{3D surface plot of $Ent^{2P}_e(S,A)$ with its maximum value obtained at the point  $(1,1,1)$ when $\beta=e$.}
\end{figure}

\begin{theorem}
    Let $(X,\mathcal{R}_1)$ and $(X,\mathcal{R}_2)$ be two approximation spaces, and $(S,A)$ be  a soft set defined over $X$.  If $\mathcal{R}_1\preceq \mathcal{R}_2$ and $\beta$ is fixed, then $(Ent^{2P}_{\beta}(S,A))_{\mathcal{R}_2} \leq (Ent^{2P}_{\beta}(S,A))_{\mathcal{R}_1} $.\\

\end{theorem}

\begin{proof}
    The proof can be obtained by following the proof of Theorem 5.2 and keeping $\beta$ fixed.
\end{proof}

\begin{theorem}
    Let $(X,\mathcal{R})$  be an approximation space,  and $(S,A)$ be  a soft set defined on $X$.  If $\theta_{\mathcal{R}}^P(S,A)\in (0,1)$,  $\theta_{\mathcal{R}}^P(C(S,A))\in (0,1)$, and $\beta_1<\beta_2$, then $Ent^{2P}_{\beta_2}(S,A) < Ent^{2P}_{\beta_1}(S,A) $.\\

\end{theorem}

\begin{proof}
Let $x= \theta_{\mathcal{R}}^P(S,A) $, $y= \theta_{\mathcal{R}}^P(C(S,A)) $, and $z_{\beta}=Ent^{2P}_{\beta}(S,A) $. So, we get $x, y\in (0,1)$. Now, 
$z_{\beta}= -\frac{1}{2}[x\log_{\beta}(\frac{x}{\beta})+ y\log_{\beta}(\frac{y}{\beta})]= \frac{x+y}{2}- \frac{1}{2\log_{e}\beta}(x\log_{e}x+y\log_{e}y)$.
Since $0<x,y<1$, then $(x\log_{e}x+y\log_{e}y)<0$. Moreover, it is given that  $\beta_1<\beta_2$. Thus,  $\frac{1}{\log_{e}{\beta_2}}<\frac{1}{\log_{e}{\beta_1}}$. Hence, $Ent^{2P}_{\beta_2}(S,A) < Ent^{2P}_{\beta_1}(S,A) $.

\end{proof}

\noindent
The following theorem investigates the fundamental results of $Ent^{P}_{\beta}(S,A)$. Its maximum value depends on $\beta$. If $\beta> e$, then the maximum value of it is obtained as 1 only when $\theta_{\mathcal{R}}^P(S,A)=\theta_{\mathcal{R}}^P(C(S,A))=1$. Figure 7 represents a 3D surface plot  of $Ent^{P}_{\beta}(S,A)$. Here, the maximum value is obtained at the point $(1,1,1)$ and is represented by the red point on the surface. On the other hand, Theorem 5.7 compares the value of $Ent^{P}_{\beta}(S,A)$ with different values of $\beta$.

\begin{theorem}
Let $(X,\mathcal{R})$ be an approximation space, and $(S,A)$ be a soft set defined over $X$. Then,  $Ent^{P}_{\beta}(S,A)$ satisfies the following results:\\

(i) The maximum value of $Ent^{P}_{\beta}(S,A)$ is given below:

\[
\max_{\substack{\theta_{\mathcal{R}}^P(S,A) \in [0,1] \\ \theta_{\mathcal{R}}^P(C(S,A)) \in [0,1]}} 
Ent^{P}_{\beta}(S,A)
=
\begin{cases}
\dfrac{\beta}{e\,\log_e \beta}, & \text{when } 1<\beta \le e \text{ and attained for } \theta_{\mathcal{R}}^P(S,A)=\theta_{\mathcal{R}}^P(C(S,A))=\dfrac{1}{\log_e \beta},\\[1.2ex]
1, & \text{when } \beta > e \text{ and attained for } \theta_{\mathcal{R}}^P(S,A)=\theta_{\mathcal{R}}^P(C(S,A))=1;
\end{cases}
\]

(ii) $Ent^{P}_{\beta}(S,A)$ is a continuous function in [0,1];\\

(iii) $Ent^{P}_{\beta}(S,A) = Ent^{P}_{\beta}(C(S,A)) $;\\

(iv) $Ent^{P}_{\beta}(S,A)$ is a concave function in [0,1]\\

\end{theorem}

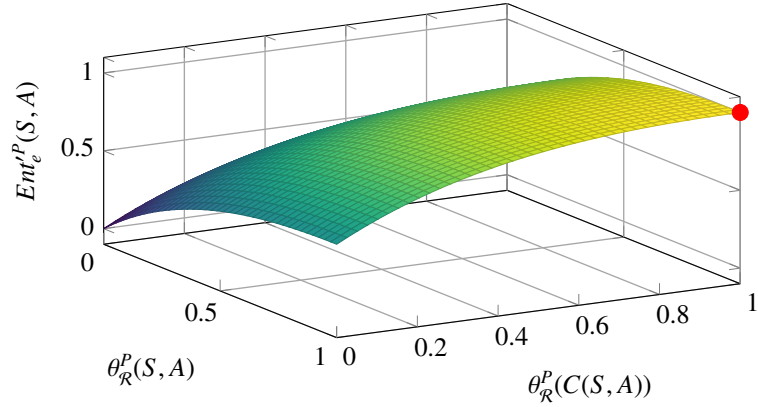
\begin{figure}[H]
\centering
\begin{tikzpicture}
\begin{axis}[
    width=10cm, height=6cm,
    view={60}{30},
    xlabel={$\theta_{\mathcal{R}}^P(S,A)$},
    ylabel={$\theta_{\mathcal{R}}^P(C(S,A))$},
    zlabel={$Ent'^{P}_e(S,A)$},
    colormap/viridis,
    domain=0:1,
    y domain=0:1,
    samples=40,
    samples y=40,
    z buffer=sort,
    grid=both,
    major grid style={line width=0.5pt,draw=gray!70},
    minor grid style={line width=0.25pt,draw=gray!30},
]

\addplot3[surf]
(
    {x},
    {y},
    {0.5*(x*exp(1-x) + y*exp(1-y))}
);

\addplot3[
    only marks,
    mark=*,
    mark size=3pt,
    color=red
] coordinates {(1,1,1)};

\end{axis}
\end{tikzpicture}
\caption{3D surface plot of $Ent^{P}_e(S,A)$ with its maximum value obtained at the point $(1,1,1)$ when $\beta = e$.}
\end{figure}

\begin{theorem}
    Let $(X,\mathcal{R})$ be an approximation space,  and $(S,A)$ be  a soft set defined over $X$.  If $\beta_1< \beta_2$, then $Ent^{P}_{\beta_1}(S,A) \leq Ent^{P}_{\beta_2}(S,A) $.\\
\end{theorem}
\begin{proof}
    The proof can be obtained easily. Thus, we skip it.
\end{proof}
\noindent
Since we discussed that the accuracy measure of a soft set in the sense of Pawlak is not following the axioms of an accuracy measure of a soft set, thus the following theorem investigates the fundamental properties of $Ent^{3P}_e(S,A)$. Figure 8 illustrates a 2D plot of $Ent^{3P}_e(S,A)$. From the figure, it is observed that the maximum value of $Ent^{3P}_e(S,A)$ is 1 and  obtained at the point  $(1/e, 1)$, while the minimum value is zero, according to the convention stated earlier.  \\

\begin{theorem}
Let $(X,\mathcal{R})$ be an approximation space, and $(S,A)$ be a soft set defined over $X$. Then,  $Ent^{3P}_e(S,A)$ satisfies the following results:\\

(i) The maximum value of $Ent^{3P}_e(S,A) $ is $1/e$ for $ \theta_{\mathcal{R}}^Y(S,A)  = 1/e$, and the minimum value is zero when $\theta_{\mathcal{R}}^Y(S,A)\in \{0,1\}$;\\

(ii) $Ent^{3P}_e(S,A) $ is a continuous function in [0,1];\\

(iii) $Ent^{3P}_e(S,A) = Ent^{3P}_e(C(S,A)) $;\\

(iv) $Ent^{3P}_e(S,A)$ is a concave function in [0,1].\\

\end{theorem}

\begin{proof}
    (iii) From Remark 4.2, we find that $\theta_{\mathcal{R}}^Y(S,A) = \theta_{\mathcal{R}}^Y(C(S,A))$. Hence, $Ent^{3P}_e(S,A) = Ent^{3P}_e(S,A) $
\end{proof}
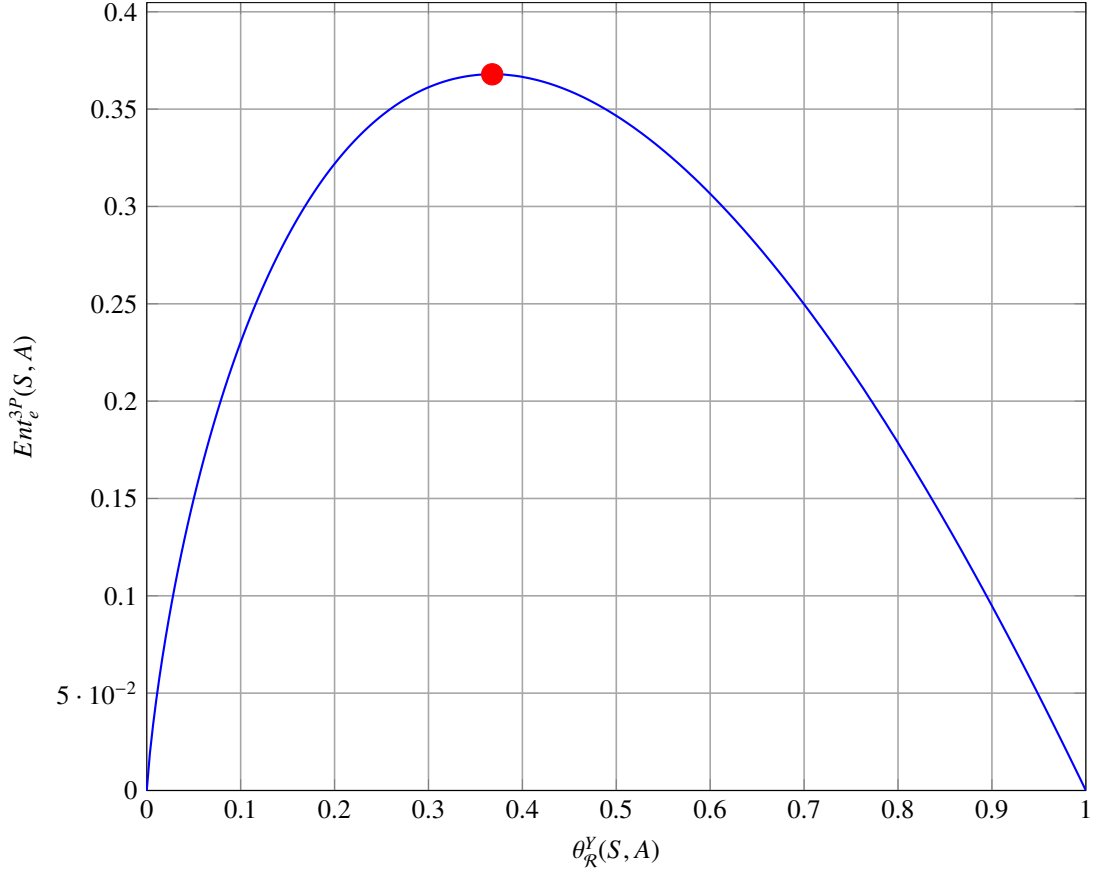
\begin{figure}[H]
\centering
\begin{tikzpicture}
\begin{axis}[
    width=14 cm,
    height=12cm,
    xmin=0, xmax=1,                 
    ymin=0,
    xlabel={$\theta_{\mathcal{R}}^Y(S,A)$},
    ylabel={$Ent^{3P}_e(S,A)$},
    grid=both,
    major grid style={line width=0.6pt,draw=gray!70},
    minor grid style={line width=0.3pt,draw=gray!40},
]

\addplot[
    domain=0:1,
    samples=300,
    thick,
    blue
]
{ (x==0 ? 0 : -x*ln(x)) };

\addplot[
    only marks,
    mark=*,
    mark size=4pt,
    red
]
coordinates {(0.36788,0.36788)};

\end{axis}
\end{tikzpicture}
\caption{2D plot of $Ent^{3P}_e(S,A)$  with its maximum value obtained at the point $(1/e,1/e)$.}
\end{figure}

\begin{theorem}
    Let $(X,\mathcal{R}_1)$ and $(X,\mathcal{R}_2)$ be two approximation spaces,  and $(S,A)$ be  a soft set defined over $X$.  If $\mathcal{R}_1\preceq \mathcal{R}_2$, then $(Ent^{3P}_e(S,A))_{\mathcal{R}_2} \leq (Ent^{3P}_e(S,A))_{\mathcal{R}_1}$.\\

\end{theorem}

\begin{proof}
    Since $\mathcal{R}_1\preceq \mathcal{R}_2$, we get  $\theta_{\mathcal{R}_1}^Y(S,A)\leq \theta_{\mathcal{R}_2}^Y(S,A)$. Hence, $\log_e(\theta_{\mathcal{R}_1}^Y(S,A))\leq \log_e(\theta_{\mathcal{R}_2}^Y(S,A))$. Thus, we have $(Ent^{3P}_e(S,A))_{\mathcal{R}_2} \leq Ent^{3P}_e(S,A))_{\mathcal{R}_1} $.
\end{proof}

\noindent 
The aforementioned theorem states that whenever the granulation of $X$ becomes finer, the corresponding entropy increases. In Theorem 5.10, the basic properties of $Ent^{4P}e(S,A)$ are investigated. Theorem 5.11 establishes that if $\beta$ is strictly increasing, then $Ent^{4P}\beta(S,A)$ is strictly decreasing. Figure 9 illustrates that the maximum value of $Ent^{4P}_\beta(S,A)$ is attained at the point $(1,1)$. We omit the proofs of the following theorems, as they can be derived using techniques similar to those employed in the aforementioned proofs.

\begin{theorem}
Let $(X,\mathcal{R})$ be an approximation space, and $(S,A)$ be a soft set defined over $X$. Then,  $Ent^{4P}_\beta(S,A)$ satisfies the following results:\\

(i) The maximum value of $Ent^{4P}_\beta(S,A) $ is given below: 
\[
\max_{\theta_{\mathcal{R}}^Y(S,A)\in[0,1]} Ent^{4P}_{\beta}(S,A)
=
\begin{cases}
\dfrac{\beta}{e\,\log_e\beta}, 
& \text{when } 1<\beta\le e,\ \text{and  attained for } \theta_{\mathcal{R}}^Y(S,A)=\dfrac{\beta}{e},\\[1.5ex]
1, 
& \text{when } \beta>e,\ \text{and attained for } \theta_{\mathcal{R}}^Y(S,A)=1;
\end{cases}
\]

(ii) $Ent^{4P}_\beta(S,A) $ is a continuous function in [0,1];\\

(iii) $Ent^{4P}_\beta(S,A) = Ent^{4P}_\beta(C(S,A)) $;\\

(iv) $Ent^{4P}_\beta(S,A)$ is a concave function in [0,1].\\

\end{theorem}

\begin{theorem}
    Let $(X,\mathcal{R})$  be an approximation space,   and $(S,A)$ be  a soft set defined over $X$.  If    $\beta_1<\beta_2$ , then $Ent^{4P}_{\beta_2}(S,A) < Ent^{4P}_{\beta_1}(S,A) $.\\

\end{theorem}

\begin{figure}[H]
\centering
\begin{tikzpicture}
\begin{axis}[
      width=14 cm,
    height=12cm,
    xmin=0, xmax=1,              
    ymin=0,
    xlabel={$\theta_{\mathcal{R}}^Y(S,A)$},
    ylabel={$Ent^{4P}_e(S,A)$},
    grid=both,
    major grid style={line width=1pt,draw=gray!80},
    minor grid style={line width=0.5pt,draw=gray!50},
]

\addplot[
    domain=0:1,                 
    samples=300,
    thick,
    blue
]
{ (x==0 ? 0 : x - x*ln(x)) };

\addplot[
    only marks,
    mark=*,
    mark size=6pt,
    color=red
]
coordinates {(1,1)};

\end{axis}
\end{tikzpicture}
\caption{2D plot of $Ent^{4P}_\beta(S,A)$ with its maximum value obtained at the point  $(1,1)$.}
\end{figure}

\noindent
Theorem 5.12 proves the basic properties of 
$Ent^{'P}_\beta(S,A)$. Also, $Ent^{'P}_\beta(S,A)$ is  strictly increasing if  $\beta$ is strictly increasing. Figure 10 shows a  2D plot of $Ent^{'P}_\beta(S,A)$ and its maximum value is obtained at the point (1,1). This point is depicted by red-color in Figure 10. \\

\begin{theorem}
Let $(X,\mathcal{R})$ be an approximation space, and $(S,A)$ be a soft set defined over $X$. Then,  $Ent^{'P}_\beta(S,A)$ satisfies the following results:\\

(i) The maximum value of $Ent^{'P}_\beta(S,A) $ is given below:\\
\[
\max_{\theta_{\mathcal{R}}^Y(S,A)\in[0,1]} Ent^{'P}_\beta(S,A) =
\begin{cases}
\dfrac{\beta}{e\,\log_e\beta}, & \text{when } 1<\beta\le e,\ \text{and attained for } \theta_{\mathcal{R}}^Y(S,A)=\dfrac{\beta}{e},\\[1.2ex]
1, & \text{when } \beta>e,\ \text{and attained for } \theta_{\mathcal{R}}^Y(S,A)=1;
\end{cases}
\], \\

(ii) $Ent^{'P}_\beta(S,A) $ is a continuous function in [0,1];\\

(iii) $Ent^{'P}_\beta(S,A) = Ent^{'P}_\beta(C(S,A)) $;\\

(iv) $Ent^{'P}_\beta(S,A)$ is a concave function in [0,1]\\

\end{theorem}

\begin{theorem}
    Let $(X,\mathcal{R})$  be an approximation space,  and $(S,A)$ be  a soft set defined over $X$.  If    $\beta_1<\beta_2$ , then $Ent^{'P}_{\beta_1}(S,A) < Ent^{'P}_{\beta_2}(S,A) $.\\

\end{theorem}

\begin{figure}[H]
\centering
\begin{tikzpicture}
\begin{axis}[
     width=14 cm,
    height=12cm,
    xlabel={$\theta_{\mathcal{R}}^Y(S,A)$},
    ylabel={$Ent^{'P}_{e}(S,A)$},
    xmin=0, xmax=1,
    ymin=0,
    grid=both,
    major grid style={line width=0.6pt,draw=gray!70},
    minor grid style={line width=0.3pt,draw=gray!40},
]

\addplot[
    domain=0:1,
    samples=300,
    thick,
    blue
]
{ x*exp(1-x) };

\addplot[
    only marks,
    mark=*,
    mark size=4pt,
    color=red
]
coordinates {(1,1)};

\end{axis}
\end{tikzpicture}
\caption{2D plot of  $Ent^{'P}_{\beta}(S,A)$ with its maximum value obtained at the point  $(1,1)$.}
\end{figure}

\section{Comparisons of six entropy measures of a soft set}

In Section 5, we introduced six types of entropy measures for a soft set. It is therefore natural to examine the significance of these different notions and to explore which one is more suitable. In this section, we present the following table to highlight the distinctions among them:  \\
\begin{table}[H]
\centering
\small
\renewcommand{\arraystretch}{1.35}
\setlength{\tabcolsep}{3pt}

\resizebox{\textwidth}{!}{
\begin{tabular}{|c|c|c|c|c|p{2.8cm}|}
\hline
\textbf{Entropy} & \textbf{Type} & \textbf{No. of Variables} & \textbf{Maximum Value} & \textbf{ Role of  $\beta$ } & \textbf{Key Features} \\
\hline

$Ent^{1P}_e$ 
& Logarithmic 
& Two 
& $1$ 
& Independent of $\beta$ 
& Continuous, symmetric, concave \\
\hline

$Ent^{2P}_\beta$ 
& Logarithmic
& Two 
& $\dfrac{\beta}{e\log_e\beta}$ or $1$ 
& Dependent on $\beta$
& Continuous, symmetric, concave, strictly increasing when $\beta$ is strictly decreasing \\
\hline

$Ent^{P}_\beta$ 
& Exponential
& Two 
& $\dfrac{\beta}{e\log_e\beta}$ or $1$ 
& Dependent on $\beta$
& Continuous, symmetric, concave, strictly increasing when $\beta$ is strictly increasing \\
\hline

$Ent^{3P}_e$ 
& Logarithmic 
& One 
& $\dfrac{1}{e}$ 
& Independent of $\beta$ 
& Continuous, symmetric, concave \\
\hline

$Ent^{4P}_\beta$ 
& Logarithmic 
& One 
& $\dfrac{\beta}{e\log_e\beta}$ or $1$ 
& Dependent on $\beta$
& Continuous, symmetric, concave, strictly increasing when $\beta$ is strictly decreasing  \\
\hline

$Ent^{'P}_\beta$ 
& Exponential
& One 
& $\dfrac{\beta}{e\log_e\beta}$ or $1$ 
& Dependent on $\beta$
& Continuous, symmetric, concave, strictly increasing when $\beta$ is strictly increasing  \\
\hline

\end{tabular}
}
\caption{Comparison of six entropy measures for a soft set}
\end{table}

\noindent
Table 2 demonstrates that each of the proposed entropy measures possesses its own mathematical importance and applicability in the study of uncertainty within soft set environments. Hence, no entropy measure can be considered universally superior in all situations. Nevertheless, among the proposed measures, $Ent^{1P}_e(S,A)$ appears to be the most mathematically consistent and theoretically stable entropy measure. This is primarily due to the fact that it is independent of the parameter $\beta$, attains a maximum value equal to $1$, and satisfies several desirable properties such as continuity, symmetry, and concavity. Moreover, its behavior closely resembles that of classical Shannon-type entropy measures, thereby making it a natural candidate for quantifying uncertainty in soft sets. \\

\noindent
The parameterized entropy measures $Ent^{2P}_{\beta}(S,A)$, $Ent^{P}_{\beta}(S,A)$, $Ent^{4P}_{\beta}(S,A)$, and $Ent^{'P}_{\beta}(S,A)$ provide additional flexibility through the parameter $\beta$, which allows one to regulate the sensitivity of the entropy with respect to uncertainty. Such measures are particularly useful in application-oriented frameworks where adaptive control over uncertainty is required. In particular, the logarithmically parameterized entropy measures exhibit strictly increasing behavior with respect to strictly decreasing values of $\beta$, whereas the exponential entropy measures exhibit strictly increasing behavior with respect to strictly increasing values of $\beta$. This contrast provides complementary perspectives for uncertainty evaluation and information discrimination. \\

\noindent
Further, the entropy measure $Ent^{1P}_e(S,A)$ is based on a two-variable structure, which makes it computationally efficient. However, its inputs do not satisfy the Yao-inspired axioms for accuracy measures of soft sets. As a result, it may not capture uncertainty with the same level of detail, particularly when one aims to study roughness measures of both the objects and their background simultaneously at an equal level in the context of image processing. Therefore, from both theoretical and structural perspectives, $Ent^{3P}_e(S,A)$ may be regarded as more suitable among these two general-purpose entropy measures. Hence, the choice of these measures is ultimately dependent on the specific requirements and judgment of the experts in the application domain.

\section{A comparison between the roughness of a soft set and the roughness of a set}

In section 3, we introduced the concepts of lower soft approximation and upper soft approximation of a soft set $(S,A)$ defined over an approximation space $(X, \mathcal{R})$. In this section, we compare our methods with existing notions of lower and upper approximations in the sense of Pawlak \cite{6}.
For this purpose, we consider Figures 1, 2, 3,  11, 12 and 13. \\

\begin{figure}[H]
    \centering
    \includegraphics[width= 10 cm, height= 10 cm]{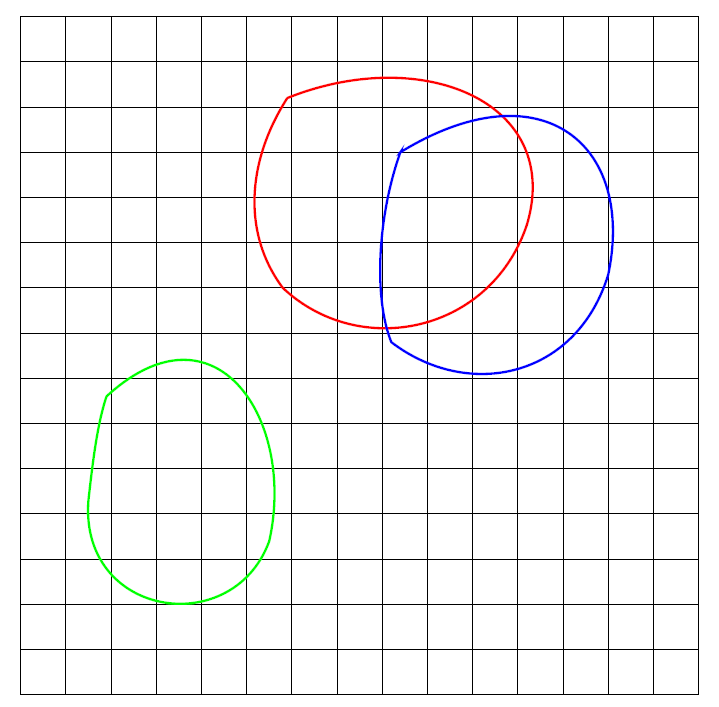} 
    \caption{ Region bounded by red color, blue color and green color in an approximation space $(X, \mathcal{R})$. }
\end{figure}

\noindent
Consider a region $P$, as shown in Figure 11, consisting of three subregions. 
These subregions are bounded by curves of three colors-red, blue, and green, respectively, 
in an approximation space $(X,\mathcal{R})$. We denote them by $R$, $B$, and $G$, 
respectively. Then, $P = R \cup B \cup G$.
If we consider Pawlak's lower approximation of the region $P$, we obtain Figure 12. 
To compute the lower approximation \cite{6}, we may partition $P$ into two regions, 
namely $R \cup B$ and $G$. However, while using Pawlak's definition of lower 
approximation \cite{6}, we do not have any theoretical flexibility to take into 
account the fact that the regions $R$ and $B$ overlap. Thus, Figure 12 represents 
the lower approximation of $P$. Similarly, using Pawlak's upper approximation of $P$, we obtain Figure 13.  In this case as well, we do not have the flexibility to consider the overlap  between the two regions R and B. Consequently, using Pawlak's definition of boundary  \cite{6}, Figure 14 is obtained. It is depicted as non-shaded regions bounded between black colored boundaries. It is important to note that Pawlak's definition of boundary \cite{6} does not  provide any provision to account for overlapping regions. Thus, it can be  understood that Pawlak's rough set theory is unable to detect the overlapping  of two objects. Hence, there is a  need for an alternative theory that can address this issue.\\

\noindent
Now, we discuss a method to resolve the issue of overlapping in Figure 11. 
Since attributes (or parameters) play an auxiliary role in soft set theory 
\cite{15,16,17,19}, we may associate three attributes $r$, $b$, and $g$ 
with the three regions $R$, $B$, and $G$, respectively. Let $A = \{r, b, g\}$ be the set of attributes and we have $P \subset X$. Define a mapping $S : A \rightarrow 2^X$ by 
$S(r) = $ region $R$, $S(b) =$ region $B$, and $S(g) = $ region $G$, respectively. Then, $(S,A)$ forms a soft set defined over  an  approximation space $(X, \mathcal{R})$. Then, using Definition 3.1. we obtain Figure 2 which depicts the lower soft approximation of $(S,A)$ determined over the approximation space $(X,\mathcal{R})$. In Figure 2, the overlapping region of regions $R$ and $G$ have some common equivalence classes of $(X,\mathcal{R})$. Similarly, using Definition 3.2, we obtain the upper soft approximation of $(S,A)$ determined over the approximation space $(X, \mathcal{R})$ and it is shown in Figure 1. Like lower soft approximation, Figure 1 also has some common equivalence classes of $(X, \mathcal{R})$. It is easy to check mathematically that these equivalence classes form the upper approximation of the common region of regions $R$ and $G$,  in the sense of Pawlak \cite{6}. Thus, using Definitions 3.3 and 3.4, we obtain Figure 3. Moreover, it is easy to find that the aforementioned problem of overlapping of two regions is now solved. This comparison concludes that our methods may be highly useful in areas like pattern recognition, image processing, AI vision, etc. 
while overlapping of objects will be there. This unique feature  open the pavement for the uses of our methodologies as a complementary methodology of rough set theory in space exploration \cite{49,50,51}.

Thus, our methodologies have the potential to advance artificial intelligence, along with several other domains, where rough set theory has played a significant role since 1982.

\begin{figure}[H]
    \centering
    \includegraphics[width= 9 cm, height= 9 cm]{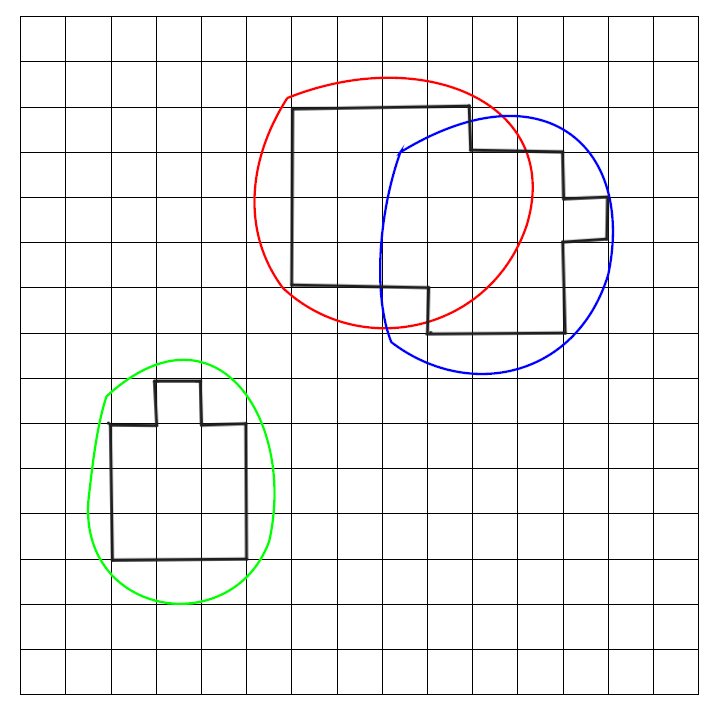} 
    \caption{ Lower approximation using Pawlak's rough set.}
\end{figure}

\begin{figure}[H]
    \centering
    \includegraphics[width= 9 cm, height= 9 cm]{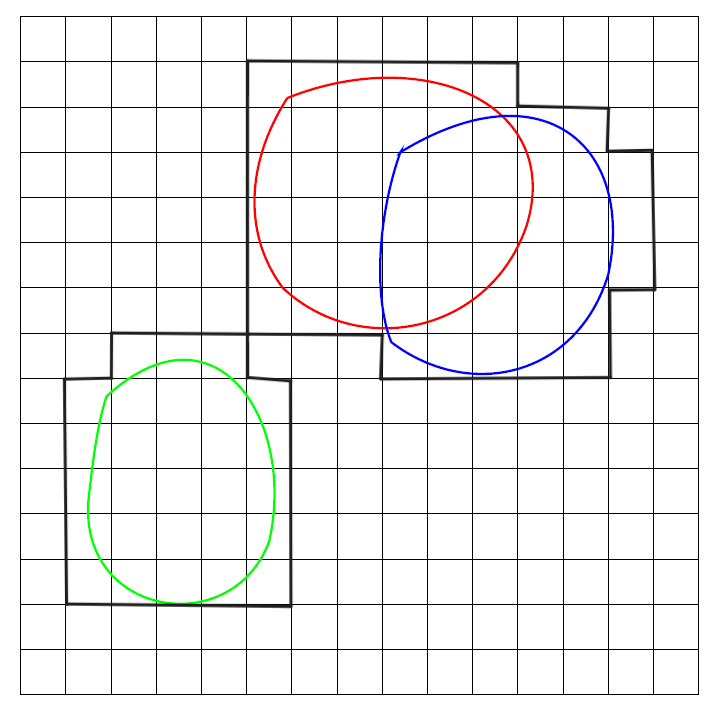} 
    \caption{Upper approximation using Pawlak's rough set.}
\end{figure}

\begin{figure}[H]
\centering
\includegraphics[width= 9 cm, height= 9 cm]{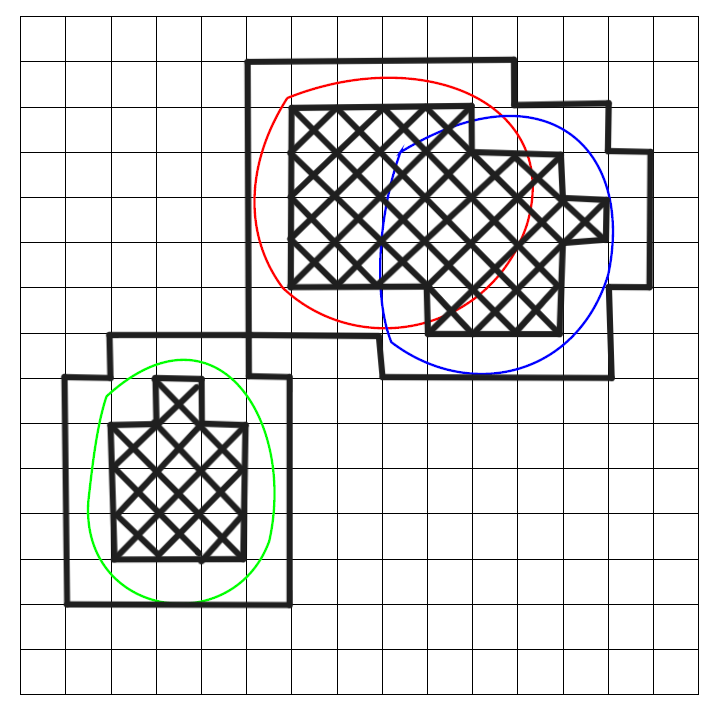}
\caption{Lower and upper approximations using of rough set theory.}

\end{figure}

\noindent

\section{Conclusion}
\noindent
Soft set theory is an important and rapidly growing area within soft computing due to its attribute-oriented mathematical framework and its wide range of applications in science and social sciences. Motivated by the theoretical limitations in selecting subsets of attribute sets, several extended and hybrid models have been developed to better capture uncertainty in complex systems. In this work, we proposed two distinct roughness measures and six entropy measures for soft sets, and systematically studied their properties using both theoretical and computational approaches. The roughness measures of soft sets are formulated within two different conceptual settings while strictly preserving the foundational principles of soft set theory introduced by Molodtsov. Based on these, six entropy measures were defined, and their structural and analytical properties were rigorously examined. A comparative study of the proposed definitions introduced in Section 3 with the existing concepts of classical rough set theory clearly establishes the novelty and theoretical importance of the proposed framework in effectively characterizing uncertainty and roughness. Furthermore, the newly introduced methodologies possess significant potential to function as efficient complementary frameworks for real time object detection and feature analysis in exoplanetary surfaces and exoplanetary images, especially when combined with advanced computational approaches such as machine learning, deep learning, and other hybrid intelligent techniques \cite{49,50,51,52,53,54,55}. From the comparative study of all proposed entropy measures, it is observed that logarithmically and exponentially parameterized measures exhibit different monotonic behaviours with respect to the parameter $\beta$, thereby providing complementary perspectives for uncertainty assessment and information discrimination. Moreover, the two-variable based entropy measure $Ent^{1P}_e(S,A)$ offers computational efficiency and simplicity; however, due to its structural limitations and the fact that its inputs do not satisfy the axioms for accuracy measures of soft sets, it may not fully capture uncertainty, particularly when both object and background roughness are required to be analyzed simultaneously at an equal level, as in image processing applications. In contrast, $Ent^{3P}_e(S,A)$ provides a more comprehensive and axiomatically consistent framework, making it more suitable for general-purpose uncertainty modeling. Finally, it is emphasized that no single entropy measure is universally optimal. Rather, the selection of an appropriate measure is inherently dependent on the specific requirements of the application and the informed judgment of domain experts. Thus, the proposed family of measures offers a flexible and adaptable toolkit for uncertainty quantification in soft set theory.\\

 \noindent
\textbf{Acknowledgment:} S. K. Pal gratefully acknowledges his ANRF Prime Minister Professorship, Govt. of India, Grant No. ANRF/PMP/2025/000027/ET.\\
\noindent
\textbf{Authors' contributions:} Theory and software: S.A.; Supervision: S.K.P; Writing: S.A.; Draft check: S.A. and S.K.P.; Final check: S.A. and S.K.P.

\end{document}